\documentclass[11pt,reqno]{amsart}
\usepackage[left=1in,right=1in,top=1in,bottom=1in]{geometry}
\usepackage{empheq}
\usepackage{enumitem}
\usepackage{amssymb,amsmath,latexsym,amsfonts,amsbsy, amsthm,mathtools,graphicx,color}
\usepackage{float}
\usepackage{hyperref}
\hypersetup{
   colorlinks=true,
   citecolor=blue,
   filecolor=blue,
   linkcolor=blue,
   urlcolor=blue
}

\def\XXint#1#2#3{{\setbox0=\hbox{$#1{#2#3}{\int}$}
  \vcenter{\hbox{$#2#3$}}\kern-.5\wd0}}

\newcommand{\al}{\alpha}

\newcommand{\lda}{\lambda}
\newcommand{\om}{\Omega}
\newcommand{\pa}{\partial}
\newcommand{\va}{\varepsilon}
\newcommand{\ud}{\mathrm{d}}
\newcommand{\be}{\begin{equation}}
\newcommand{\ee}{\end{equation}}

\newcommand{\Lda}{\Lambda}

\newcommand{\A}{\mathbf{A}}
\newcommand{\cA}{\mathcal{A}}

\newcommand{\cE}{\mathcal{E}}

\newcommand{\I}{\mathbf{I}}

\newcommand{\Z}{\mathbb{Z}}

\newcommand{\m}{\mathbf{m}}

\newcommand{\n}{\mathbf{n}}

\newcommand{\OO}{\mathbb{O}}

\newcommand{\R}{\mathbb{R}}

\newcommand{\Ss}{\mathbb{S}}

\newcommand{\uu}{\mathbf{u}}

\newcommand{\vv}{\mathbf{v}}

\newcommand{\W}{\mathbf{W}}

\newcommand{\wc}{\rightharpoonup}

\newcommand{\HH}{\mathcal{H}}

\newcommand{\vp}{\varphi}

\newcommand{\T}{\mathrm{T}}
\newcommand{\ga}{\gamma}
\newcommand{\Ga}{\Gamma}

\newcommand{\ift}{\infty}
\newcommand{\wt}{\widetilde}
\newcommand{\wh}{\widehat}
\newcommand{\f}{\frac}

\newcommand{\ol}{\overline}
\newcommand{\Ra}{\Rightarrow}
\newcommand{\op}{\operatorname}

\newcommand{\na}{\nabla}

\DeclareMathOperator{\vol}{vol}
\DeclareMathOperator{\tr}{tr}

\DeclareMathOperator{\argmin}{argmin}

\def\<{\langle}\def\>{\rangle}
\def\({\left(}\def\){\right)}
\numberwithin{equation}{section}
\theoremstyle{plain}
\newtheorem{thm}{Theorem}[section]

\newtheorem{lem}[thm]{Lemma}
\newtheorem{prop}[thm]{Proposition}

\theoremstyle{definition}
\newtheorem{defn}[thm]{Definition}

\theoremstyle{remark}
\newtheorem{rem}[thm]{Remark}
\newtheorem{claim}[thm]{Claim}

\date{}

\title[Weak solutions for matrix-valued harmonic map flows]{Existence of weak solutions for two-phase \\ matrix-valued harmonic map flows}

\author{Wei Wang}
\address{School of Mathematical Sciences, Peking University, Beijing 100871, China}
\email{gjmtamag@gmail.com,\,\,2201110024@stu.pku.edu.cn}

\author{Wei Wang}
\address{School of Mathematical Sciences, Zhejiang University, Hangzhou 310058, China}
\email{wangw07@zju.edu.cn}

\author{Zhifei Zhang}
\address{School of Mathematical Sciences, Peking University, Beijing 100871, China}
\email{zfzhang@math.pku.edu.cn}

\begin{document}

\begin{abstract}
We study the existence of weak solutions to a two-phase matrix-valued harmonic
map flow with lifespan determined by that of the prescribed interface. This
system arises as the limiting equation for the matrix-valued Keller--Rubinstein--Sternberg problem studied in \cite{FLWZ23}. Our construction is based on a modified minimizing movement scheme. Precisely, we discretize time and build approximate solutions by interpolating minimizers of the associated variational problem on each time step.
\end{abstract}

\maketitle

\section{Introduction}

Let $\om \subset \R^d$ be a bounded domain with smooth boundary, where $d \geq 2$. We study the following matrix-valued evolution system, which couples a heat flow with a geometric motion of an interface.
\begin{subequations}\label{HeatSystems}
\begin{empheq}[left=\empheqlbrace]{align}
\quad &\pa_t \A_{\pm} = \Delta \A_{\pm} - \sum_{\al=1}^d \pa_{\al}\A_{\pm}\A_{\pm}^{\T}\pa_{\al}\A_{\pm} && \text{in } \om_{\pm}(t), \label{HeatFlow}\\
&\A_+ \in \OO_+(n), \quad \A_- \in \OO_-(n) && \text{in } \om_{\pm}(t), \label{ValueOfA}\\
&\A_+^{\T}\A_- \in \mathbb{P}_n && \text{on } \Ga(t), \label{MinimalPair}\\
&\pa_{\nu}\A_+ = \pa_{\nu}\A_- && \text{on } \Ga(t), \label{NeumannJump}\\
&V= \kappa && \text{on } \Ga(t). \label{Interface}
\end{empheq}
\end{subequations}
The notations in the above system are given as bellowing:
\begin{itemize}
\item $\OO_{\pm}(n)$ denote the sets of orthogonal matrices in $\mathbb{M}_n$ with determinants $\pm 1$. We write $\OO(n):=\OO_+(n)\cup \OO_-(n)$.

\item $\mathbb{P}_n := \{\I - 2\n \otimes \n : |\n| = 1,\, \n \in \R^n\}$ is the set of reflection matrices in $\R^n$. For vectors $\n,\m \in \R^n$, the tensor product is defined by $(\n \otimes \m)_{ij} = \n_i \m_j$.

\item The condition \eqref{MinimalPair} means that $(\A_+,\A_-)$ forms a minimal pair across the interface. This property is symmetric: $(\A_+,\A_-)$ is a minimal pair if and only if $(\A_-,\A_+)$ is.

\item $\Ga(t) := \pa \om_+(t) \cap \pa \om_-(t)$ denotes the moving interface that separates the subdomains $\om_+(t)$ and $\om_-(t)$.

\item $V$, $\kappa$, and $\nu$ denote the normal velocity, the mean curvature, and the unit normal vector of $\Ga(t)$, respectively, where $\nu$ points from $\om_+(t)$ to $\om_-(t)$. Equation \eqref{Interface} states that $\Ga(t)$ evolves by mean curvature flow.
\end{itemize}

\subsection{Background and problem setting}

The system \eqref{HeatSystems} originates from the study of the matrix-valued
Keller--Rubinstein--Sternberg problem. This framework, introduced in
\cite{RSK89a,RSK89b}, concerns the asymptotic behavior of a phase-indicator map
$u_{\va}:(U\subset\R^d)\times[0,+\ift)\to\R^n$ that solves the system
\be
\pa_t u_{\va}-\Delta u_{\va}=\va^{-2}\pa_u F(u_{\va}).\label{RSKmodel}
\ee
Here, $F=F(u):\R^n\to[0,+\ift)$ is a smooth potential that vanishes in two
connected, disjoint submanifolds of $\R^n$. The original
Keller--Rubinstein--Sternberg model takes the form
\[
\pa_t v_{\va}-\va\Delta v_{\va}=\va^{-1}\pa_u F(v_{\va}),
\]
and describes chemical reaction-diffusion processes governed by a small
parameter $0<\va\ll 1$. In this formulation, the reaction term
$\va^{-1}\pa_u F(v_{\va})$ dominates, while the diffusion term
$\va\Delta v_{\va}$ is comparatively weak. Then $u_{\va}(x,t)=v_{\va}(x,\va t)$ satisfies \eqref{RSKmodel}. A canonical example
is the Allen--Cahn equation, where $n=1$ and
\[
F(u)=\f{1}{4}(1-u^2)^2
\]
is a double-well potential. More broadly, the Keller--Rubinstein--Sternberg
framework serves as a multi-phase extension of the Ginzburg--Landau model, in
which $F$ is a non-negative potential that vanishes on a single connected manifold.
We refer the reader to \cite{BK91,Che92,ESS92,LW99,LW02a,LW02b} for an
extensive treatment of both models.

When $n\geq 2$, the dynamics of \eqref{RSKmodel} become substantially more
complex, even in the static case; see \cite{FLWZ23,FWZZ15,FWZZ18,LPW12} for
further developments. Fei--Lin--Wang--Zhang \cite{FLWZ23} analyzed the asymptotic behavior of the
dynamic system
\be
\pa_t\A_{\va}=\Delta\A_{\va}-\va^{-2}(\A_{\va}\A_{\va}^{\T}\A_{\va}-\A_{\va}),
\quad\A_{\va}:\om\times[0,+\ift)\to\mathbb{M}_n,\label{RSKmodelmatrix}
\ee
which is the gradient flow of the energy functional
\[
\cE_{\va}(\A,\na\A)=\int_{\om}\(\f{1}{2}\|\na\A\|^2
+\f{1}{\va^2}F(\A)\)\ud x,\quad F(\A):=\f{1}{4}\|\A^{\T}\A-\I\|^2.
\]
Here, $ \I=\I_n $ denotes the identity matrix in $\mathbb{M}_n$, and $\|\cdot\|$
is the norm on $\mathbb{M}_n$ induced by the Frobenius inner product
\[
\mathbf{X}:\mathbf{Y}=\sum_{i,j=1}^n\mathbf{X}_{ij}\mathbf{Y}_{ij},
\quad\mathbf{X},\mathbf{Y}\in\mathbb{M}_n.
\]

The authors of \cite{FLWZ23} derived the limiting system \eqref{HeatSystems} as $\va\to 0^+$ (see also \cite{Liu25} for related equations). This is a two-phase problem consisting of a pair of harmonic map flows \eqref{HeatFlow} in the time-dependent subdomains $\om_{\pm}(t)$, coupled through the interface conditions \eqref{MinimalPair}, \eqref{NeumannJump}, and \eqref{Interface} on $\Ga(t)$.

Recall that in the one-phase setting, Chen and Struwe \cite{CS89} established a convergence result for the Ginzburg--Landau approximation to the harmonic map flow. They proved that solutions of the approximating system converge to a weak solution of the harmonic map flow, and that the convergence is smooth away from a concentration set of lower parabolic Hausdorff dimension. However, for the two-phase system, the corresponding forward convergence problem is substantially more difficult because of the added complexity of the transition between the two phases.

For this reason, \cite{FLWZ23} studied the problem from the opposite
direction. More precisely, they proved that if \eqref{HeatSystems} admits a local smooth solution, then there exists a family $ \{\A_{\va}\}_{\va>0}$ of solutions to \eqref{RSKmodelmatrix} that converges to this smooth solution as $\va\to 0^+$. This in turn raises the problem of constructing such local smooth solutions. However, in the present two-phase setting, the strong nonlinearity of the interface conditions \eqref{MinimalPair}--\eqref{NeumannJump} makes the local smooth solvability of \eqref{HeatSystems} much more difficult than that of the classical one-phase model. From this perspective, the global existence of weak solutions established in the present paper should be viewed as an alternative approach to understanding the well-posedness of \eqref{HeatSystems}.

\subsection{Weak solutions}

Before stating our main theorems, we define weak solutions of
\eqref{HeatSystems} and record several remarks that clarify our framework. In
this paper, we write $ \mathbb{S}_n $ and $ \mathbb{A}_n $ for the spaces of
symmetric and anti-symmetric $ n\times n $ matrices, respectively.

\begin{defn}[Weak solution]
For a given $ T>0 $ and a smooth family of hypersurfaces $\{\Ga(t)\}_{t\in[0,T]}$ satisfying the mean curvature flow law \eqref{Interface}, we let
\[
\om_{T,\pm}:=\bigcup_{t\in(0,T)}\om_{\pm}(t),\quad\Ga_T:=\bigcup_{t\in(0,T)}\Ga(t).
\]
 We say that $ \A=\A_{\pm}:\om_{T,\pm}\to\OO_{\pm}(n) $ is a weak solution of \eqref{HeatSystems} if the following properties hold:
\begin{enumerate}[label=$(\theenumi)$]
\item $ \pa_t\A_{\pm}\in L^2(0,T;L^2(\om_{\pm}(\cdot))) $ and
$ \na\A_{\pm}\in L^{\ift}(0,T;L^2(\om_{\pm}(\cdot))) $.
\item $ (\A_+,\A_-) $ satisfies the minimal pair condition
\eqref{MinimalPair} on $ \Ga_T $ in the sense of trace.
\item For any $ \Phi=\Phi_{\pm}\in C_0^{\ift}(\om_{T,\pm},\mathbb{M}_n) $,
$ \A=\A_{\pm} $ satisfies the interior weak formula
\be
\int_{0}^T\int_{\om_{\pm}(t)}(\pa_t\A_{\pm}\Phi_{\pm}+\na\A_{\pm}\na\Phi_{\pm}
+\na\A_{\pm}\A_{\pm}^{\T}\na\A_{\pm}\Phi_{\pm})\ud x\ud t=\mathbf{O},
\label{weakformula}
\ee
where $ \mathbf{O}=\mathbf{O}_n $ is the zero matrix in $ \mathbb{M}_n $.
\item For any $ \Psi=\Psi_{\pm}\in C_0^{\ift}(\om_{T,\pm}\cup\Ga_T,\mathbb{A}_n) $
such that
\be
\Psi_+=\Psi_-,\quad\HH^{d}\text{-a.e. on }\Ga_T,\label{BoundaryTest}
\ee
$ \A=\A_{\pm} $ satisfies the weak Neumann jump condition
\begin{align}
\sum_{\zeta=\pm}\int_0^T\int_{\om_{\zeta}(t)}\((\A_{\zeta}^{\T}\pa_t\A_{\zeta}):\Psi_{\zeta}+(\A_{\zeta}^{\T}\na\A_{\zeta}):\na\Psi_{\zeta}\)\ud x\ud t&=0,\label{WeakNeumann}\\
\sum_{\zeta=\pm}\int_0^T\int_{\om_{\zeta}(t)}\((\pa_t\A_{\zeta}\A_{\zeta}^{\T}):\Psi_{\zeta}+(\na\A_{\zeta}\A_{\zeta}^{\T}):\na\Psi_{\zeta}\)\ud x\ud t&=0.\label{WeakNeumann1}
\end{align}
\end{enumerate}
\end{defn}

\begin{rem}\label{InterfaceGiven}
The evolution of $\Ga(t)$ (i.e. \eqref{Interface}) is actually the well-known mean curvature flow, which is independent of $\A_\pm(t)$ and can be decoupled with \eqref{HeatFlow}-\eqref{NeumannJump}.  Therefore, in our setting, $ \Ga(t) $ is not an unknown free boundary produced by the limit $ \va\to 0 $. Rather, the interface is prescribed in advance. The main goal of this paper is to study the wellposedness of the system for $(\A_+, \A_-)$ for given $\Gamma(t)$ which evolves according to the mean curvature flow.

\end{rem}

\begin{rem}\label{EquivalenceRemarks}
Equations
\eqref{WeakNeumann} and \eqref{WeakNeumann1} provide a weak formulation of the following Neumann jump conditions:
\begin{align*}
  \A_+^{\T}\pa_{\nu}\A_+=\A_-^{\T}\pa_{\nu}\A_-,\qquad \pa_{\nu}\A_+\A_+^{\T}=\pa_{\nu}\A_-\A_-^{\T},
\end{align*}
which is equivalent to \eqref{NeumannJump} under condition \eqref{MinimalPair}. We refer to Lemma \ref{lemNeumann1} for details.
\end{rem}

\begin{rem}\label{RelationWeak}
Next, we explain that \eqref{WeakNeumann} implies \eqref{weakformula}. Fix
$ \Phi=\Phi_{\pm}\in C_0^{\ift}(\om_{T,\pm},\mathbb{M}_n) $. For any
$ \mathbf{X}=\mathbf{X}_{\pm}\in C_0^{\ift}(\om_{T,\pm},\mathbb{M}_n) $, we
decompose $ \mathbf{X} $ into its symmetric and anti-symmetric parts by
\[
\mathbf{X}=\f{\mathbf{X}+\mathbf{X}^{\T}}{2}+\f{\mathbf{X}-\mathbf{X}^{\T}}{2}
:=\mathbf{X}^0+\mathbf{X}^1.
\]
Then
\[
\mathbf{X}_{\pm}^0\in C_0^{\ift}(\om_{T,\pm},\mathbb{S}_n),\quad
\mathbf{X}_{\pm}^1\in C_0^{\ift}(\om_{T,\pm},\mathbb{A}_n).
\]
Since $ \A_{\pm}\in\OO_{\pm}(n) $, we have that
$ \A_{\zeta}^{\T}\pa_t\A_{\zeta} $ and
$ \A_{\zeta}^{\T}\pa_{\al}\A_{\zeta} $ are anti-symmetric for
$ \al\in\Z\cap[1,d] $. Therefore, \eqref{WeakNeumann} gives
\be
\begin{aligned}
&\sum_{\zeta=\pm}\int_{0}^{T}\int_{\om_{\zeta}(t)}
\((\A_{\zeta}^{\T}\pa_t\A_{\zeta}):\mathbf{X}_{\zeta}
+(\A_{\zeta}^{\T}\na\A_{\zeta}):\na\mathbf{X}_{\zeta}\)\ud x\ud t\\
&\quad=
\sum_{i=0,1}\sum_{\zeta=\pm}\int_{0}^{T}\int_{\om_{\zeta}(t)}
\((\A_{\zeta}^{\T}\pa_t\A_{\zeta}):\mathbf{X}_{\zeta}^i
+(\A_{\zeta}^{\T}\na\A_{\zeta}):\na\mathbf{X}_{\zeta}^i\)\ud x\ud t,
\end{aligned}\label{eq5}
\ee
where \eqref{BoundaryTest} is automatic because $ \mathbf{X} $ is compactly
supported. Let $ \phi=\phi_{\pm}\in C_0^{\ift}(\om_{T,\pm},\R) $. For
$ 1\leq i,j\leq n $, let $ \mathbf{E}_{ij} $ denote the $ n\times n $-matrix
whose entry $ (i,j) $ is $ 1 $ and whose other entries are $ 0 $. Taking
\[
\mathbf{X}=\mathbf{E}_{ij}\phi
\]
as a test function in \eqref{eq5}, we obtain that for any $ i,j\in\Z\cap[1,n] $,
\[
\sum_{\zeta=\pm}\int_{0}^{T}\int_{\om_{\zeta}(t)}
\left[(\pa_t\A_{\zeta}\A_{\zeta}^{\T})_{ij}\phi+(\na\A_{\zeta}\A_{\zeta}^{\T})_{ij}\na\phi\right]\ud x\ud t=0.
\]
Fix $ i,j,k $. Using a standard approximation argument, we may take
$ \phi=(\A_{\pm}\Phi_{\pm})_{jk} $ and deduce that
\[
\sum_{\zeta=\pm}\int_{0}^{T}\int_{\om_{\zeta}(t)}
\left\{(\pa_t\A_{\zeta}\A_{\zeta}^{\T})_{ij}(\A_{\zeta}\Phi_{\zeta})_{jk}
+(\na\A_{\zeta}\A_{\zeta}^{\T})_{ij}
\na[(\A_{\zeta}\Phi_{\zeta})_{jk}]\right\}\ud x\ud t=0.
\]
Summing over $ j $ from $ 1 $ to $ n $, it follows that for any
$ i,k\in\Z\cap[1,n] $,
\[
\sum_{\zeta=\pm}\sum_{j=1}^n\int_{0}^{T}\int_{\om_{\zeta}(t)}
\left\{(\pa_t\A_{\zeta}\A_{\zeta}^{\T})_{ij}(\A_{\zeta}\Phi_{\zeta})_{jk}
+(\na\A_{\zeta}\A_{\zeta}^{\T})_{ij}
\na[(\A_{\zeta}\Phi_{\zeta})_{jk}]\right\}\ud x\ud t=0,
\]
which implies \eqref{weakformula}.
\end{rem}

\subsection{Main results}

We now state our main existence theorems. We first introduce the class of
admissible initial data. Let $ \Ga=\pa\om_+\cap\pa\om_- $ and
$ \om=\om_+\cup\om_-\cup\Ga $. We define
\be
\mathcal{A}_{\Ga}
=\{\A=\A_{\pm}\in H^1(\om_{\pm},\OO_{\pm}(n)):\A_+^{\T}\A_-\in\mathbb{P}_n
\text{ on }\Ga\}.\label{FunctionalSet}
\ee

We begin with the case in which the interface is stationary.

\begin{thm}[Existence of weak solutions, $\op{I}$]\label{ExistenceFix}
Let $ \om=\om_+\cup\om_-\cup\Ga\subset\R^d $ be a smooth domain. Assume that
$ \Ga(t)=\Ga $ is a stationary minimal hypersurface. Then for any $ T>0 $ and
any $ \A_0=\A_{0,\pm}\in\mathcal{A}_{\Ga} $, there exists a weak solution
$ \A=\A_{\pm}:\om_{T,\pm}=\om_{\pm}\times(0,T)\to\OO_{\pm}(n) $ of
\eqref{HeatSystems} with initial trace $ \A(0)=\A_0 $ such that
\be
\lim_{t\to 0}\|\A_{\pm}(\cdot,t)-\A_{0,\pm}(\cdot)\|_{L^2(\om_{\pm},\mathbb{M}_n)}=0.\label{initial1}
\ee
\end{thm}

We next turn to the case of a moving interface.

\begin{thm}[Existence of weak solutions, $\op{II}$]\label{ExistenceNonFix}
Let $ \om=\om_+(0)\cup\om_-(0)\cup\Ga(0)\subset\R^d $ be a smooth domain. Assume that $ \Ga(t) $ evolves smoothly and has lifespan $ T_0>0 $, in the sense that $ \Ga(t) $ remains regular on $ [0,T_0) $. Then for any $ 0<T<T_0 $ and any $ \A_0=\A_{0,\pm}\in\mathcal{A}_{\Ga(0)} $, there exists a weak solution $ \A=\A_{\pm}:\om_{T,\pm}\to\OO_{\pm}(n) $ of \eqref{HeatSystems} with initial trace $ \A(0)=\A_0 $. In particular, for $ t>0 $ sufficiently small, there is a diffeomorphism $ X_{\pm}(\cdot,t):\om_{\pm}(0)\to\om_{\pm}(t) $ and
\be
\lim_{t\to 0}\sum_{\zeta=\pm}\|\A_{\zeta}(X_{\pm}(\cdot,t),t)-\A_{0,\zeta}(\cdot)\|_{L^2(\om_{\zeta}(0),\mathbb{M}_n)}=0.\label{initial2}
\ee
\end{thm}

\begin{rem}
The weak solutions constructed in Theorems \ref{ExistenceFix} and
\ref{ExistenceNonFix} are not expected to be unique. This is already the case for the one-phase harmonic map flow, where the non-uniqueness of weak solutions was shown in \cite{Cor90}. The same phenomenon should therefore be expected for the system \eqref{HeatSystems}.
\end{rem}

\begin{rem}
The weak solutions constructed in Theorems \ref{ExistenceFix} and
\ref{ExistenceNonFix} exist for the full lifespan of the prescribed interface.
In Theorem \ref{ExistenceFix}, the interface is stationary, and hence a weak solution exists for every $ T>0 $. In Theorem \ref{ExistenceNonFix}, the interface $ \Ga(t) $ exists smoothly up to time $ T_0 $, and the weak solution is constructed on each interval $ (0,T) $ with $ 0<T<T_0 $. In this sense, the lifespan of the weak solution is limited only by that of the evolving
interface.
\end{rem}

\subsection{Difficulties, strategies, and novelty}

Establishing the existence of weak solutions to the matrix-valued two-phase harmonic map flow \eqref{HeatSystems} involves substantial difficulties. In the one-phase setting, Chen and Struwe \cite{CS89} constructed weak solutions using a Ginzburg--Landau approximation. However, for the present two-phase problem, this strategy is not directly available. The evolving interface $ \Ga(t) $ and the coupling conditions on $ \Ga(t) $, in particular the minimal pair condition \eqref{MinimalPair}, make the interaction between the two phases much more difficult to control. As a result, it is difficult to derive the convergence of solutions $ \A_{\va} $ of the Keller--Rubinstein--Sternberg model \eqref{RSKmodelmatrix} directly to a solution of \eqref{HeatSystems}.

As noted above, \cite{FLWZ23} instead considered the reverse problem: assuming
the existence of a sufficiently regular solution of the limiting system, they
constructed approximating solutions $ \A_{\va} $ of \eqref{RSKmodelmatrix} that converge to it. This further indicates the difficulty of applying Ginzburg--Landau-type methods directly in the present setting. In particular, the construction of a Chen--Struwe-type weak solution of \eqref{HeatSystems} by passing to the limit in \eqref{RSKmodelmatrix} remains open.

To overcome this difficulty, we use a minimizing movement scheme, following the abstract framework of \cite{AGS08} and the concrete example in \cite{LX16}. We discretize the time interval, solve a variational problem on each time step, and then interpolate the discrete solutions to obtain an approximate evolution. Passing to the limit as the time step tends to zero yields a weak solution.

Before turning to the full two-phase problem, we first discuss a toy model for
global weak heat flows into spheres. This one-phase example illustrates the
basic idea of the minimizing movement scheme and clarifies the changes needed in the proof of the two-phase case.

\subsubsection{A toy model: harmonic map flows into unit spheres}

To explain the minimizing movement scheme in a simpler setting, we first
consider the harmonic map flow into the unit sphere $ \Ss^{n-1}\subset\R^n $, where $ n\geq 2 $. Let $ (M,g) $ be a compact smooth Riemannian manifold of dimension $d\geq 2$ without boundary, and let $ \uu_0\in H^1(M,\Ss^{n-1}) $. We study the initial value problem
\be
\left\{\begin{aligned}
&\pa_t\uu=\Delta_g\uu+|\na\uu|^2\uu&\text{in }&M\times\R_+,\\
&\uu(\cdot,0)=\uu_0&\text{on }&M,
\end{aligned}\right.\label{HeatFlowSimple}
\ee
for a map $ \uu:M\times\R_+\to\Ss^{n-1} $. We will construct weak solutions using
a discrete-time minimizing movement scheme. The result is as follows.

\begin{prop}\label{proptoymodel}
For $ \uu_0\in H^1(M,\Ss^{n-1}) $, there exists a weak solution
$ \uu:M\times\R_+\to\Ss^{n-1} $ of \eqref{HeatFlowSimple} such that
$ \na\uu\in L^{\ift}(\R_+;L^2(M)) $ and $ \pa_t\uu\in L^2(\R_+;L^2(M)) $. In particular,
\be
\lim_{t\to 0^+}\|\uu(\cdot,t)-\uu_0\|_{L^2(M)}=0.\label{initialdateSd1}
\ee
\end{prop}

Here, by a weak solution, we mean that for any
$ \vp\in C_0^{\ift}(M\times\R_+) $,
\[
\int_{M\times\R_+}\(\pa_t\uu\vp+\<\na\uu,\na\vp\>_g\)\ud\vol_g\ud t
=\int_{M\times\R_+}|\na\uu|^2\uu\vp\ud\vol_g\ud t.
\]
In local coordinates, $ \ud\vol_g=\sqrt{\det g}\ud x $. We begin with the
discrete variational problem
\be
\min_{\uu\in H^1(M,\Ss^{n-1})}E_h(\wt{\uu};\uu):=\min_{\uu\in H^1(M,\Ss^{n-1})}\int_{M}\(|\na\uu|^2+2h^{-1}|\wt{\uu}-\uu|^2\)\ud\vol_g,\label{Functional}
\ee
where $ h>0 $ and $ \wt{\uu}\in H^1(M,\Ss^{n-1}) $. We first prove the existence of
a minimizer.

\begin{prop}\label{TheoremExistence}
Assume that $ h>0 $ and $ \wt{\uu}\in H^1(M,\Ss^{n-1}) $. Then there exists a minimizer $ \uu\in H^1(M,\Ss^{n-1}) $ for \eqref{Functional}.
\end{prop}

\begin{proof}
Set
\[
\ga:=\inf_{\vv\in H^1(M,\Ss^{n-1})}E_h(\wt{\uu};\vv),
\]
and choose a minimizing sequence $ \{\uu_k\}\subset H^1(M,\Ss^{n-1}) $ such
that $ E_h(\wt{\uu};\uu_k)\to\ga $ as $ k\to+\ift $. Then
$ \{\uu_k\} $ is bounded in $ H^1(M,\R^n) $. Passing to a subsequence, we may
assume that
\[
\uu_k\wc\uu\quad\text{weakly in }H^1(M,\R^n).
\]
By the Rellich theorem, after passing to a further subsequence if necessary,
$ \uu_k\to\uu $ strongly in $ L^2(M,\R^n) $ and a.e. on $ M $. Since
$ |\uu_k|=1 $ a.e. on $ M $, it follows that $ |\uu|=1 $ a.e. on $ M $, and
hence $ \uu\in H^1(M,\Ss^{n-1}) $. By weak lower semicontinuity,
\be
\ga\leq E_h(\wt{\uu};\uu)\leq\liminf_{k\to+\ift}E_h(\wt{\uu};\uu_k)=\ga.
\ee
Therefore, $ \uu $ is a minimizer of \eqref{Functional}.
\end{proof}

We now construct the discrete solutions and pass to the limit.

\begin{proof}[Proof of Proposition \ref{proptoymodel}]
To construct weak solutions of \eqref{HeatFlowSimple} on $ M\times\R_+ $, it
is enough to prove existence on $ M\times(0, T) $ for arbitrary $ T>0 $, with
bounds independent of $ T $ on compact time intervals. Fix $ T>0 $, let
$ N\in\Z_+ $, and set $ h=\f{T}{N} $. The proof is divided into several steps.

\begin{enumerate}[label=$\text{Step }\theenumi.$]
\item Set $ \uu_h^0:=\uu_0\in H^1(M,\Ss^{n-1}) $.

\item For each $ k=0,1,\dots,N-1 $, let $ \uu_h^{k+1} $ be a minimizer of
\be
\min_{\vv\in H^1(M,\Ss^{n-1})}E_h(\uu_h^k;\vv),\label{minimizingProperty}
\ee
whose existence follows from Proposition \ref{TheoremExistence}.

We next derive the Euler--Lagrange equation. Let $ \vv $ be a minimizer of
\eqref{Functional}. For any $ \vp\in C^{\ift}(M,\R^n) $, consider the variation
$ \Pi_{\Ss^{n-1}}(\vv+\va\vp) $, where
$ \Pi_{\Ss^{n-1}}:N_{\delta}(\Ss^{n-1})\to\Ss^{n-1} $ is the nearest-point
projection from a tubular neighborhood of $ \Ss^{n-1} $ onto $ \Ss^{n-1} $.
Differentiating at $ \va=0 $, we obtain
\[
-\Delta_g\vv+\f{1}{h}(\vv-\wt{\uu})\perp T_{\vv(x)}\Ss^{n-1}.
\]
Since $ |\vv|=1 $, this is equivalent to
\[
-\Delta_g\vv+\f{1}{h}(\vv-\wt{\uu})=\(\left|\na\vv\right|^2+\f{1-\vv\wt{\uu}}{h}\)\vv.
\]
Hence, for each $ k\in\Z\cap[1,N] $, we have
\be
-\Delta_g \uu_h^k+\f{1}{h}(\uu_h^k-\uu_h^{k-1})
=\(\left|\na \uu_h^k\right|^2+\f{1-\uu_h^k\uu_h^{k-1}}{h}\)\uu_h^k.
\label{DiscreteEquation}
\ee

\item Let $ t_k=kh $ for $ k\in\Z\cap[0,N] $. We define the piecewise linear
and piecewise constant interpolations by
\be
\left\{\begin{aligned}
\wt{\uu}_N(x,t)&:=\f{t-t_{k-1}}{h}\uu_h^k(x)+\(1-\f{t-t_{k-1}}{h}\)\uu_h^{k-1}(x),&x\in M,\ t\in(t_{k-1},t_k],\\
\ol{\uu}_N(x,t)&:=\uu_h^k(x),&x\in M,\ t\in(t_{k-1},t_k].
\end{aligned}\right.
\ee
Then
\[
\pa_t\wt{\uu}_N=\f{\uu_h^k-\uu_h^{k-1}}{h}\quad
\text{for }t\in(t_{k-1},t_k].
\]
By \eqref{DiscreteEquation}, it follows that
\be
\pa_t\wt{\uu}_N=\Delta_g\ol{\uu}_N+\(\left|\na \uu_h^k\right|^2+\f{1-\uu_h^k\uu_h^{k-1}}{h}\)\ol{\uu}_N\quad\text{in }M\times(0,T),
\label{Approximate}
\ee
in the sense of distributions.
\end{enumerate}

We next derive uniform estimates. Since $ |\uu_h^k|=1 $ a.e. on $ M $, the
minimality of $ \uu_h^k $ gives
\[
E_h(\uu_h^{k-1};\uu_h^k)\leq E_h(\uu_h^{k-1};\uu_h^{k-1})
=\int_M|\na \uu_h^{k-1}|^2\ud\vol_g,
\]
and therefore
\[
2h^{-1}\int_M|\uu_h^k-\uu_h^{k-1}|^2\ud\vol_g
+\int_M|\na\uu_h^k|^2\ud\vol_g\leq\int_M|\na\uu_h^{k-1}|^2\ud\vol_g.
\]
Equivalently,
\[
2h\int_M\left|\f{\uu_h^k-\uu_h^{k-1}}{h}\right|^2\ud\vol_g
+\int_M|\na\uu_h^k|^2\ud\vol_g\leq\int_M|\na\uu_h^{k-1}|^2\ud\vol_g.
\]
Summing over $ k=1,\dots,N $, we obtain
\[
2\int_{M\times(0,T)}\left|\pa_t\wt{\uu}_N\right|^2\ud\vol_g\ud t
+\sup_{t\in(0,T)}\int_M|\na\ol{\uu}_N|^2\ud\vol_g
\leq\int_M|\na\uu_0|^2\ud\vol_g.
\]
Since $ \wt{\uu}_N $ is a convex combination of $ \uu_h^k $ and
$ \uu_h^{k-1} $, the same bound yields
\[
\sup_{t\in(0,T)}\int_M|\na\wt{\uu}_N|^2\ud\vol_g
\leq C\int_M|\na\uu_0|^2\ud\vol_g.
\]
Hence
\begin{gather*}
\{\ol{\uu}_N\}\text{ is bounded in }L^{\ift}(0,T;H^1(M,\R^n)),\\
\{\wt{\uu}_N\}\text{ is bounded in }L^{\ift}(0,T;H^1(M,\R^n)),\\
\{\pa_t\wt{\uu}_N\}\text{ is bounded in }L^2(M\times(0,T);\R^n).
\end{gather*}
By compactness, after passing to a subsequence, there exist
$ \uu,\vv:M\times(0,T)\to\R^n $ such that
\be
\begin{gathered}
\wt{\uu}_N\to \uu\text{ strongly in }L^2(M\times(0,T),\R^n),\\
\na\wt{\uu}_N\wc^*\na\uu\text{ weakly* in }L^{\ift}(0,T;L^2(M,\R^n)),\\
\pa_t\wt{\uu}_N\wc\pa_t\uu\text{ weakly in }L^2(M\times(0,T),\R^n),\\
\ol{\uu}_N\wc\vv\text{ weakly in }L^2(M\times(0,T),\R^n),\\
\na\ol{\uu}_N\wc^* \na\vv\text{ weakly* in }L^{\ift}(0,T;L^2(M,\R^n)).
\end{gathered}\label{convergenceu}
\ee

We now show that $ \uu=\vv $. For $ t\in(t_{k-1},t_k] $,
\[
\wt{\uu}_N-\ol{\uu}_N=\(\f{t-t_{k-1}}{h}-1\)(\uu_h^k-\uu_h^{k-1}),
\]
and hence
\[
|\wt{\uu}_N-\ol{\uu}_N|^2\leq|\uu_h^k-\uu_h^{k-1}|^2.
\]
Integrating over $ M\times(t_{k-1},t_k] $ and summing in $ k $, we obtain
\[
\int_{M\times(0,T)}|\wt{\uu}_N-\ol{\uu}_N|^2\ud\vol_g\ud t
\leq h^2\int_{M\times(0,T)}|\pa_t\wt{\uu}_N|^2\ud\vol_g\ud t\to 0,
\]
as $ N\to+\ift $. Thus $ \uu=\vv $ a.e. on $ M\times(0,T) $. Since $ |\ol{\uu}_N|=1 $ a.e., it follows that $ |\uu|=1 $ a.e. on $ M\times(0,T) $.

To identify the limit equation, we use the following characterization of weak solutions.

\begin{lem}[\cite{LW08}, Lemma 7.5.4]\label{lemLW08}
A map $ \uu:M\times(0,T)\to\Ss^{n-1} $ with
$ \na\uu\in L^{\ift}(0,T;L^2(M,\R^n)) $ and
$ \pa_t\uu\in L^2(M\times(0,T);\R^n) $ is a weak solution of
\eqref{HeatFlowSimple} if and only if
\[
\pa_t\uu\wedge\uu-\op{div}(\na\uu\wedge\uu)=0
\quad\text{in }M\times(0,T),
\]
where $ \wedge $ denotes the wedge product in $ \R^n $.
\end{lem}

We now pass to the limit in \eqref{Approximate}. Wedge \eqref{Approximate} with
$ \ol{\uu}_N\phi $, where $ \phi\in C_0^{\ift}(M\times(0,T)) $. Since
$ \ol{\uu}_N\wedge\ol{\uu}_N=0 $, we obtain
\[
\int_{M\times(0,T)}(\pa_t\wt{\uu}_N-\Delta_g\ol{\uu}_N)\wedge
(\ol{\uu}_N\phi)\ud\vol_g\ud t=0.
\]
Integrating by parts gives
\be
\int_{M\times(0,T)}
\(\pa_t\wt{\uu}_N\wedge\ol{\uu}_N\phi
+\na\ol{\uu}_N\wedge\ol{\uu}_N\cdot\na\phi\)\ud\vol_g\ud t=0.\label{wedgeequality}
\ee
Since $ \ol{\uu}_N\to\uu $ strongly in $ L^2 $ while $ \pa_t\wt{\uu}_N\wc\pa_t\uu $ weakly in $ L^2 $, we may pass to the limit in the first term of \eqref{wedgeequality}. Likewise, since $ \ol{\uu}_N\to\uu $ strongly in $ L^2 $ and $ \na\ol{\uu}_N\wc \na\uu $ weakly in $ L^2 $, we may pass to the limit in the second term. Hence
\[
\int_{M\times(0,T)}
\(\pa_t\uu\wedge\uu\,\phi+\na\uu\wedge\uu\cdot\na\phi\)\ud\vol_g\ud t=0,
\]
that is,
\[
\pa_t\uu\wedge\uu-\op{div}(\na\uu\wedge\uu)=0
\quad\text{in }M\times(0,T).
\]
By Lemma \ref{lemLW08}, $ \uu $ is a weak solution of
\eqref{HeatFlowSimple} on $ M\times(0,T) $.

Finally, the construction gives
$ \wt{\uu}_N(\cdot,0)=\uu_0 $, and the bound on $ \pa_t\wt{\uu}_N $ implies equicontinuity in $ L^2(M) $. Passing to the limit yields
\eqref{initialdateSd1}. Since $ T>0 $ is arbitrary, the proof is complete.
\end{proof}

\subsubsection{\eqref{HeatSystems} vs. the toy model: difficulties and adaptations}

The toy model above illustrates the minimizing movement scheme for a one-phase harmonic map flow into the sphere $ \Ss^{n-1} $. Our problem \eqref{HeatSystems} is more difficult in several respects. We now explain the main differences and the corresponding changes in the argument.

The first difference concerns the target manifold. In the toy model, the target is the sphere $ \Ss^{n-1} $, whereas here it is $ \OO_{\pm}(n) $, the set of orthogonal matrices with determinant $ \pm 1 $. The geometry of
$ \OO_{\pm}(n) $ is different from that of the sphere, so the arguments used in the toy model do not apply directly. In particular, the wedge-product identity used there must be replaced by an analog adapted to matrix-valued maps. This is one of the basic structural changes in the proof.

The second difference is that the interface $ \Ga(t) $ moves by mean curvature flow. In the toy model, the domain is fixed, but here the domains $ \om_{\pm}(t) $ vary with time. This makes the discrete variational problem more delicate since the functional must compare configurations defined on different domains at different time steps. To address this point, we add a correction term to the functional (see \eqref{FunctionalProblem}). This term is designed to absorb the motion of the interface and make the minimizing movement scheme compatible with the evolving interface.

The third difference comes from the interface conditions. In contrast to the boundary-free toy model, the system \eqref{HeatSystems} involves the coupling conditions on $ \Ga(t) $, especially the minimal pair condition \eqref{MinimalPair}. Although the minimizing movement scheme produces approximations that are compatible with these constraints at the discrete level, passing to the limit is not straightforward. The regularity of the piecewise constant interpolation $ \ol{u}_N $ seems too weak for some parts of the argument. In contrast, piecewise linear interpolation $ \wt{u}_N $ alone does not provide sufficient information about the constraints of the interface. The proof, therefore, requires both families of approximations. Precisely, we use their complementary properties to pass to the limit and to show that the resulting weak solution satisfies the interface conditions.

\section{Functional problems}

We introduce the functional
\[
E_h(\mathbf{V},\wt{\A};\A) := \sum_{\zeta=\pm}\int_{\om_{\zeta}}\left(\|\na\A_{\zeta}\|^2 + h^{-1}\|\A_{\zeta}-\wt{\A}_{\zeta}\|^2 + 2(\mathbf{V}_{\zeta}\cdot\na\wt{\A}_{\zeta}):(\A_{\zeta}-\wt{\A}_{\zeta})\right)\ud x,
\]
which depends on three parameters:
\begin{itemize}
\item $h>0$ is a positive constant.
\item $\wt{\A}\in\cA_{\Ga}$, where $\Ga$ and $\cA_{\Ga}$ are as in \eqref{FunctionalSet}.
\item $\mathbf{V}:=\mathbf{V}_{\pm}=(\mathbf{V}_{\pm}^{\al})\in L^{\ift}(\om_{\pm},\R^n)$, with the notation
\[
\mathbf{V}_{\pm}\cdot\na\wt{\A}_{\pm} := \sum_{\al=1}^d\mathbf{V}_{\pm}^{\al}\pa_{\al}\wt{\A}_{\pm}.
\]
\end{itemize}
We study the minimization problem
\be
\min_{\A=\A_{\pm}\in\cA_{\Ga}}E_h(\mathbf{V},\wt{\A};\A).\label{FunctionalProblem}
\ee
We adopt the convention $E_h(0^d,\wt{\A};\A)=:E_h(\wt{\A};\A)$, where $0^d\in\R^d$ is the zero vector.

As a preliminary step, analogous to Proposition~\ref{TheoremExistence}, we establish the existence of a minimizer for \eqref{FunctionalProblem}.

\begin{prop}\label{FA}
Given $\wt{\A}=\wt{\A}_{\pm}\in\cA_{\Ga}$, $\mathbf{V}=\mathbf{V}_{\pm}\in L^{\ift}(\om_{\pm},\R^n)$, and $h>0$, the minimization problem \eqref{FunctionalProblem} admits a minimizer
\[
\A:=\A(\mathbf{V},\wt{\A};h)=\mathop{\argmin}_{\mathbf{B}\in\cA_{\Ga}}E_h(\mathbf{V},\wt{\A};\mathbf{B})\in\cA_{\Ga}.
\]
For brevity, we write $\A(0^d,\wt{\A};h)=:\A(\wt{\A};h)$.
\end{prop}

\begin{proof}
Let $\{\A^{(k)}\}=\{\A_{\pm}^{(k)}\}_{k\geq 1}$ be a minimizing sequence:
\[
\lim_{k\to+\ift}E_h(\mathbf{V},\wt{\A};\A^{(k)})=\inf_{\A=\A_{\pm}\in\cA_{\Ga}}E_h(\mathbf{V},\wt{\A};\A)=:\ga.
\]
We first show $\ga\in\R$. Evaluating at $\A=\wt{\A}$ gives
\[
E_h(\mathbf{V},\wt{\A};\wt{\A})=\sum_{\zeta=\pm}\int_{\om_{\zeta}}\|\na\wt{\A}_{\zeta}\|^2\,\ud x<+\ift,
\]
so $\ga<+\ift$. For the lower bound, Young's inequality yields, for any $\A=\A_{\pm}\in\cA_{\Ga}$,
\begin{align*}
&\sum_{\zeta=\pm}\int_{\om_{\zeta}}\left(\|\na\A_{\zeta}\|^2+h^{-1}\|\A_{\zeta}-\wt{\A}_{\zeta}\|^2+2(\mathbf{V}_{\zeta}\cdot\na\wt{\A}_{\zeta}):(\A_{\zeta}-\wt{\A}_{\zeta})\right)\ud x\\
&\quad\geq\sum_{\zeta=\pm}\int_{\om_{\zeta}}\left(\|\na\A_{\zeta}\|^2+(2h)^{-1}\|\A_{\zeta}-\wt{\A}_{\zeta}\|^2-2h\|\mathbf{V}_{\zeta}\|_{\ift}^2\|\na\wt{\A}_{\zeta}\|^2\right)\ud x\\
&\quad\geq-\sum_{\zeta=\pm}\int_{\om_{\zeta}}2h\|\mathbf{V}_{\zeta}\|_{\ift}^2\|\na\wt{\A}_{\zeta}\|^2\,\ud x,
\end{align*}
so $\ga>-\ift$. Hence $\ga\in\R$.

We claim that
\be
\sup_{k\in\Z_+}\|\A_{\pm}^{(k)}\|_{H^1(\om_{\pm},\mathbb{M}_n)}<+\ift.\label{supclaim}
\ee
For sufficiently large $k$, the lower bound above gives
\[
2\ga\geq E_h(\mathbf{V},\wt{\A};\A^{(k)})\geq\sum_{\zeta=\pm}\int_{\om_{\zeta}}\left(\|\na\A_{\zeta}^{(k)}\|^2+(2h)^{-1}\|\A_{\zeta}^{(k)}-\wt{\A}_{\zeta}\|^2-2h\|\mathbf{V}_{\zeta}\|_{\ift}^2\|\na\wt{\A}_{\zeta}\|^2\right)\ud x,
\]
and hence
\[
\|\na\A_{\pm}^{(k)}\|_{L^2(\om_{\pm},\mathbb{M}_n)}^2\leq 2\ga+2h\|\mathbf{V}_{\pm}\|_{\ift}^2\|\na\wt{\A}_{\pm}\|_{L^2(\om_{\pm},\mathbb{M}_n)}^2.
\]
Since elements of $\OO_{\pm}(n)$ are uniformly bounded, claim \eqref{supclaim} follows.

By \eqref{supclaim} and the reflexivity of $H^1(\om_{\pm},\mathbb{M}_n)$, there exist $\A=\A_{\pm}\in H^1(\om_{\pm},\mathbb{M}_n)$ and a subsequence (not relabeled) such that
\[
\A_{\pm}^{(k)}\rightharpoonup\A_{\pm}\quad\text{weakly in }H^1(\om_{\pm},\mathbb{M}_n).
\]
By the Rellich--Kondrachov theorem, $H^1(\om_{\pm},\mathbb{M}_n)\hookrightarrow L^2(\om_{\pm},\mathbb{M}_n)$ is compact, so the convergence is also strong in $L^2(\om_{\pm},\mathbb{M}_n)$. Passing to a further subsequence, $\A_{\pm}^{(k)}\to\A_{\pm}$ pointwise a.e.; since $\OO_{\pm}(n)$ is closed in $\mathbb{M}_n$ and each $\A_{\pm}^{(k)}(x)\in\OO_{\pm}(n)$ a.e., we conclude $\A_{\pm}(x)\in\OO_{\pm}(n)$ for a.e.\ $x\in\om_{\pm}$. The weak lower semicontinuity of $E_h$ then gives
\[
E_h(\mathbf{V},\wt{\A};\A)\leq\liminf_{k\to+\ift}E_h(\mathbf{V},\wt{\A};\A^{(k)})=\inf_{\A=\A_{\pm}\in\cA_{\Ga}}E_h(\mathbf{V},\wt{\A};\A),
\]
so $\A$ achieves the infimum. It remains to verify that $\A$ satisfies the boundary condition \eqref{MinimalPair}.

By the trace theorem, weak convergence in $H^1(\om_{\pm},\mathbb{M}_n)$ implies weak convergence in $H^{\f{1}{2}}(\pa\om_{\pm},\mathbb{M}_n)$. The compact embedding $H^{\f{1}{2}}(\pa\om_{\pm})\hookrightarrow L^2(\pa\om_{\pm})$ then gives a strong convergence $\A_{\pm}^{(k)}\to\A_{\pm}$ in $L^2(\pa\om_{\pm},\mathbb{M}_n)$. Since $\A_{\pm}^{(k)}$ and $\A_{\pm}$ take values in $\OO_{\pm}(n)$ and are thus uniformly bounded on $\Ga$, we estimate
\begin{align*}
&\|(\A_+^{(k)})^{\T}\A_-^{(k)}-\A_+^{\T}\A_-\|_{L^2(\Ga)}\\
&\quad\leq\|(\A_+^{(k)})^{\T}\|_{L^{\ift}(\Ga)}\|\A_-^{(k)}-\A_-\|_{L^2(\Ga)}+\|\A_-\|_{L^{\ift}(\Ga)}\|(\A_+^{(k)})^{\T}-\A_+^{\T}\|_{L^2(\Ga)}\to 0
\end{align*}
as $k\to+\ift$. Passing to a further subsequence, $(\A_+^{(k)})^{\T}\A_-^{(k)}\to\A_+^{\T}\A_-$ a.e.\ on $\Ga$. Since each $(\A_+^{(k)})^{\T}\A_-^{(k)}\in\mathbb{P}_n$ a.e.\ and $\mathbb{P}_n$ is closed, we conclude $\A_+^{\T}\A_-\in\mathbb{P}_n$ a.e.\ on $\Ga$, so $\A\in\cA_{\Ga}$.
\end{proof}

Having established the existence, we derive the Euler--Lagrange equations characterizing every minimizer of \eqref{FunctionalProblem}.

\begin{lem}\label{EulerLagrange}
Suppose $\A=\A_{\pm}$ is a minimizer of \eqref{FunctionalProblem}. For any $\W=\W_{\pm}\in C_0^{\ift}(\om_{\pm}\cup\Ga,\mathbb{A}_n)$ satisfying $\W_+=\W_-$ $\HH^{d-1}$-a.e.\ on $\Ga$, we have
\begin{align}
\sum_{\zeta=\pm}\int_{\om_{\zeta}}\left((\A_{\zeta}^{\T}\na\A_{\zeta}):\na\W_{\zeta}+\f{\A_{\zeta}^{\T}(\A_{\zeta}-\wt{\A}_{\zeta})}{h}:\W_{\zeta}+(\mathbf{V}_{\zeta}\cdot\na\wt{\A}_{\zeta}):(\A_{\zeta}\W_{\zeta})\right)\ud x=0,\label{WeakEulerLagrange}\\
\sum_{\zeta=\pm}\int_{\om_{\zeta}}\left((\na\A_{\zeta}\A_{\zeta}^{\T}):\na\W_{\zeta}+\f{(\A_{\zeta}-\wt{\A}_{\zeta})\A_{\zeta}^{\T}}{h}:\W_{\zeta}+(\mathbf{V}_{\zeta}\cdot\na\wt{\A}_{\zeta}):(\W_{\zeta}\A_{\zeta})\right)\ud x=0.\label{WeakEulerLagrange1}
\end{align}
\end{lem}

\begin{proof}
Let $\W_{\pm}\in C_0^{\ift}(\om_{\pm}\cup\Ga,\mathbb{A}_n)$ satisfy $\W_+=\W_-$ on $\Ga$. Since $\W_{\pm}$ is antisymmetric, $\exp(t\W_{\pm})\in\OO_{\pm}(n)$ for each $t\in\R$. Moreover, the pair $(\A_+\exp(t\W_+),\A_-\exp(t\W_-))$ is admissible: writing $\W:=\W_{\pm}|_{\Ga}$, we have
\[
(\A_+\exp(t\W))^{\T}(\A_-\exp(t\W))=\exp(-t\W)\A_+^{\T}\A_-\exp(t\W)\in\mathbb{P}_n\quad\text{a.e.\ on }\Ga,
\]
and $ \det(\A_{\pm}\exp(tW))=\pm $. The minimizing property of $\A_{\pm}$ gives
\[
E_h(\mathbf{V},\wt{\A};\A\exp(t\W))\geq E_h(\mathbf{V},\wt{\A};\A)\quad\text{for all }t\in\R,
\]
so $t=0$ is a minimizer of $t\mapsto E_h(\mathbf{V},\wt{\A};\A\exp(t\W))$. Differentiating at $t=0$ and dividing by $2$ yields
\begin{align*}
0=\sum_{\zeta=\pm}\int_{\om_{\zeta}}\left(\na\A_{\zeta}:\na(\A_{\zeta}\W_{\zeta})+h^{-1}(\A_{\zeta}-\wt{\A}_{\zeta}):(\A_{\zeta}\W_{\zeta})+(\mathbf{V}_{\zeta}\cdot\na\wt{\A}_{\zeta}):(\A_{\zeta}\W_{\zeta})\right)\ud x.
\end{align*}
Using the product rule $\na(\A_{\zeta}\W_{\zeta})=(\na\A_{\zeta})\W_{\zeta}+\A_{\zeta}(\na\W_{\zeta})$ and the algebraic identities
\be
\begin{aligned}
&\na\A_{\pm}:((\na\A_{\pm})\W_{\pm})=0,\quad\na\A_{\pm}:(\A_{\pm}\na\W_{\pm})=(\A_{\pm}^{\T}\na\A_{\pm}):\na\W_{\pm},\\
&(\A_{\pm}-\wt{\A}_{\pm}):(\A_{\pm}\W_{\pm})=(\A_{\pm}^{\T}(\A_{\pm}-\wt{\A}_{\pm})):\W_{\pm},
\end{aligned}\label{Frobidentity}
\ee
we deduce \eqref{WeakEulerLagrange}. Here, the first identity holds because for each $\al$, $(\pa_{\al}\A_{\pm})^{\T}(\pa_{\al}\A_{\pm})$ is symmetric and $\W_{\pm}\in\mathbb{A}_n$ is antisymmetric; the second and third follows from the property of the Frobenius product. To obtain \eqref{WeakEulerLagrange1}, the argument is repeated using $\exp(t\W_{\pm})\A_{\pm}$ as the admissible competitor.
\end{proof}

\section{Constructions of weak solutions}

\subsection{The fixed-interface case: proof of Theorem \ref{ExistenceFix}}

We first consider the simpler case where the interface does not move, that is,
$ \Ga(t)=\Ga $ for all $ t $, with $ \Ga $ a stationary minimal hypersurface.
We first prove Theorem \ref{ExistenceFix} in this setting, since the argument is cleaner and already contains the main ideas used in the general case.

\subsubsection{Construction of approximating solutions}

Throughout this subsection, we assume that $ \om_{\pm}(t)=\om_{\pm} $. Let
$ \A_0=\A_{0,\pm}\in\mathcal{A}_{\Ga} $ be the initial data. Fix $ T>0 $,
$ N\in\Z_+ $, and set $ h:=\f{T}{N} $. We construct approximate solutions
$ \A_h=\A_{h,\pm} $ and $ \ol{\A}_h=\ol{\A}_{h,\pm} $ on
$ \om_{\pm}\times(0,T) $ as follows.

\begin{enumerate}[label=$\text{Step }\theenumi.$]
\item Set $ \A_h^0:=\A_h^0{}_{\pm}:=\A_{0,\pm} $.

\item We define a sequence inductively
$ \{\A_h^m\}_{m=0}^N:=\{\A_{h,\pm}^m\}_{m=0}^N\subset\mathcal{A}_{\Ga} $. Assume that $ \A_h^m $ has already been constructed. We then let
$ \A_h^{m+1}:=\A(\A_h^m;h)\in\mathcal{A}_{\Ga} $ be a minimizer of
\eqref{FunctionalProblem} with $ \mathbf{V}\equiv 0^d $. The existence of
$ \A_h^{m+1} $ follows from Proposition \ref{FA}.

\item For each $ m\in\Z\cap[0,N] $, define $ \A_h(\cdot,mh):=\A_h^m(\cdot) $. For $ x\in\om_{\pm} $ and $ t\in(mh,(m+1)h] $, set
\begin{align*}
\A_h(x,t)&:=\f{t-mh}{h}\A_h^{m+1}(x)+\(1-\f{t-mh}{h}\)\A_h^m(x),\\
\ol{\A}_h(x,t)&:=\A_h^{m+1}(x).
\end{align*}
\end{enumerate}

\subsubsection{Uniform estimates and convergence of \texorpdfstring{$ \A_h $}{} and \texorpdfstring{$ \ol{\A}_h $}{}}

We next derive the estimates needed to pass to the limit. By minimizing
the property of $ \A_h^{m+1} $, we have
\[
E_h(\A_h^m;\A_h^{m+1})\leq E_h(\A_h^m;\A_h^m).
\]
In particular,
\be
\sum_{\zeta=\pm}\int_{\om_{\zeta}}
\(\|\na\A_{h,\zeta}^{m+1}\|^2+h^{-1}\|\A_{h,\zeta}^{m+1}-\A_{h,\zeta}^m\|^2\)\ud x
\leq
\sum_{\zeta=\pm}\int_{\om_{\zeta}}\|\na\A_{h,\zeta}^m\|^2\ud x.
\label{Comparison}
\ee
Since
\[
\pa_t\A_h=\f{1}{h}(\A_h^{m+1}-\A_h^m)
\quad\text{for }t\in(mh,(m+1)h],
\]
it follows that
\be
\sum_{\zeta=\pm}\int_{\om_{\zeta}}\|\pa_t\A_{h,\zeta}(\cdot,t)\|^2\ud x
=\sum_{\zeta=\pm}\int_{\om_{\zeta}}h^{-2}\|\A_{h,\zeta}^{m+1}-\A_{h,\zeta}^m\|^2\ud x
\label{GeodesicEstimate}
\ee
for any $ t\in(mh,(m+1)h] $. Combining \eqref{Comparison} and
\eqref{GeodesicEstimate}, we obtain
\[
\sum_{\zeta=\pm}\int_{mh}^{(m+1)h}\int_{\om_{\zeta}}
\|\pa_t\A_{h,\zeta}(\cdot,t)\|^2\ud x\ud t
\leq
\sum_{\zeta=\pm}
\(\int_{\om_{\zeta}}\|\na\A_{h,\zeta}^m\|^2\ud x
-\int_{\om_{\zeta}}\|\na\A_{h,\zeta}^{m+1}\|^2\ud x\).
\]
Summing over $ m=0,\dots,N-1 $, and using \eqref{Comparison} once more, we get
\be
\sup_{0\leq t\leq T}
\(\sum_{\zeta=\pm}\int_{\om_{\zeta}}\|\na\ol{\A}_{h,\zeta}\|^2\ud x\)
+\sum_{\zeta=\pm}\int_0^T\int_{\om_{\zeta}}
\|\pa_t\A_{h,\zeta}\|^2\ud x\ud t\leq\sum_{\zeta=\pm}\int_{\om_{\zeta}}\|\na\A_{0,\zeta}\|^2\ud x.
\label{UniformEstimates}
\ee
Therefore,
\begin{gather*}
\{\ol{\A}_h\}\text{ is bounded in }L^{\ift}(0,T;H^1(\om_{\pm},\mathbb{M}_n)),\\
\{\A_h\}\text{ is bounded in }L^{\ift}(0,T;H^1(\om_{\pm},\mathbb{M}_n)),\\
\{\pa_t\A_h\}\text{ is bounded in }L^2(0,T;L^2(\om_{\pm},\mathbb{M}_n)).
\end{gather*}
Hence, there exist maps $ \A=\A_{\pm},\mathbf{B}=\mathbf{B}_{\pm}:\om_{\pm}\times(0,T)\to\mathbb{M}_n $
such that
\begin{align*}
&\A\in L^2(\om_{\pm}\times(0,T),\mathbb{M}_n),\quad
\na\A\in L^{\ift}(0,T;L^2(\om_{\pm},\mathbb{M}_n)),\\
&\pa_t\A\in L^2(\om_{\pm}\times(0,T),\mathbb{M}_n),\quad
\mathbf{B}\in L^{\ift}(0,T;H^1(\om_{\pm},\mathbb{M}_n)),
\end{align*}
and up to a subsequence as $ h\to 0^+ $,
\be
\begin{gathered}
\A_h\to\A\text{ strongly in }L^2(\om_{\pm}\times(0,T),\mathbb{M}_n),\\
\na\A_h\wc^*\na\A\text{ weakly* in }L^{\ift}(0,T;L^2(\om_{\pm},\mathbb{M}_n)),\\
\pa_t\A_h\wc\pa_t\A\text{ weakly in }L^2(\om_{\pm}\times(0,T),\mathbb{M}_n),\\
\ol{\A}_h\wc^*\mathbf{B}\text{ weakly* in }L^{\ift}(0,T;H^1(\om_{\pm},\mathbb{M}_n)).
\end{gathered}
\label{Convergence1}
\ee

We now compare $ \A_h $ and $ \ol{\A}_h $. By definition,
\[
\begin{aligned}
\A_h(x,t)-\ol{\A}_h(x,t)
&=
\f{t-mh}{h}\A_h^{m+1}(x)+\(1-\f{t-mh}{h}\)\A_h^m(x)-\A_h^{m+1}(x)\\
&=
\f{t-(m+1)h}{h}\(\A_h^{m+1}(x)-\A_h^m(x)\),
\end{aligned}
\]
for $ x\in\om_{\pm} $ and $ t\in(mh,(m+1)h] $. Consequently,
\[
\|\A_{h,\pm}(x,t)-\ol{\A}_{h,\pm}(x,t)\|^2=(t-(m+1)h)^2\left\|\f{\A_{h,\pm}^{m+1}(x)-\A_{h,\pm}^m(x)}{h}\right\|^2.
\]
Using \eqref{UniformEstimates}, for any $ t\in(mh,(m+1)h] $, we obtain
\[
\begin{aligned}
\sum_{\zeta=\pm}\int_{\om_{\zeta}}\|\A_{h,\zeta}-\ol{\A}_{h,\zeta}\|^2\ud x
&=
(t-(m+1)h)^2
\sum_{\zeta=\pm}\int_{\om_{\zeta}}
\left\|\f{\A_{h,\zeta}^{m+1}-\A_{h,\zeta}^m}{h}\right\|^2\ud x\\
&\leq
h\sum_{\zeta=\pm}\int_{\om_{\zeta}}\|\na\A_{0,\zeta}\|^2\ud x.
\end{aligned}
\]
Therefore,
\[
\sum_{\zeta=\pm}\int_{mh}^{(m+1)h}\int_{\om_{\zeta}}
\|\A_{h,\zeta}-\ol{\A}_{h,\zeta}\|^2\ud x\ud t
\leq
h^2\sum_{\zeta=\pm}\int_{\om_{\zeta}}\|\na\A_{0,\zeta}\|^2\ud x.
\]
Summing over $ m=0,\dots,N-1 $, we get
\be
\sum_{\zeta=\pm}\int_0^T\int_{\om_{\zeta}}
\|\A_{h,\zeta}-\ol{\A}_{h,\zeta}\|^2\ud x\ud t
\leq Th\sum_{\zeta=\pm}\int_{\om_{\zeta}}\|\na\A_{0,\zeta}\|^2\ud x.
\label{AAOL}
\ee
Letting $ h\to 0^+ $ and using \eqref{Convergence1}, we conclude that
\be
\ol{\A}_h\to\A\text{ strongly in }L^2(\om_{\pm}\times(0,T),\mathbb{M}_n).
\label{ConOlAA}
\ee
Hence $ \mathbf{B}=\A $ holds a.e. on $ \om_{\pm}\times(0,T) $. Since
\[
\ol{\A}_h\in L^{\ift}(0,T;H^1(\om_{\pm},\OO_{\pm}(n))),
\]
we may pass to a further subsequence and assume that the convergence in
\eqref{ConOlAA} holds a.e. It follows that
\[
\A=\mathbf{B}:\om_{\pm}\times(0,T)\to\OO_{\pm}(n).
\]

\subsubsection{Proof that \texorpdfstring{$ (\A_+,\A_-) $}{} is a minimal pair}

We next show that the limit preserves the minimal pair condition on the
interface. Fix $ 0<\lda<1 $ and define
\[
\A_{\lda,h}(x,t):=\left\{
\begin{aligned}
&\f{(m+1-\lda)h-t}{(1-\lda)h}\A_h^m(x)
+\f{t-mh}{(1-\lda)h}\A_h^{m+1}(x),
&&t\in(mh,(m+1-\lda)h],\\
&\A_h^{m+1}(x),
&&t\in((m+1-\lda)h,(m+1)h].
\end{aligned}
\right.
\]
Then
\[
\pa_t\A_{\lda,h}(x,t)=\left\{
\begin{aligned}
&\f{1}{(1-\lda)h}\(\A_h^{m+1}(x)-\A_h^m(x)\),
&&t\in(mh,(m+1-\lda)h],\\
&0,
&&t\in((m+1-\lda)h,(m+1)h],
\end{aligned}
\right.
\]
and
\[
\A_{\lda,h}(x,t)-\ol{\A}_h(x,t)=
\left\{
\begin{aligned}
&\f{t-(m+1-\lda)h}{(1-\lda)h}\(\A_h^{m+1}(x)-\A_h^m(x)\),
&&t\in(mh,(m+1-\lda)h],\\
&0,
&&t\in((m+1-\lda)h,(m+1)h].
\end{aligned}
\right.
\]
It follows that
\begin{align*}
\|\pa_t\A_{\lda,h,\pm}(x,t)\|
&\leq\f{1}{1-\lda}\|\pa_t\A_{h,\pm}(x,t)\|,\\
\|\na\A_{\lda,h,\pm}(x,t)\|
&\leq \|\na\A_{h,\pm}^m(x)\|+\|\na\A_{h,\pm}^{m+1}(x)\|,\\
\|\A_{\lda,h,\pm}(x,t)-\ol{\A}_{h,\pm}(x,t)\|
&\leq\f{\max\{1-\lda,\lda\}h}{1-\lda}
\left\|\f{\A_{h,\pm}^{m+1}(x)-\A_{h,\pm}^m(x)}{h}\right\|.
\end{align*}
Using \eqref{Comparison}, \eqref{UniformEstimates}, and \eqref{AAOL}, we obtain
\begin{gather*}
\{\A_{\lda,h}\}\text{ is bounded in }L^{\ift}(0,T;H^1(\om_{\pm},\mathbb{M}_n)),\\
\{\pa_t\A_{\lda,h}\}\text{ is bounded in }L^2(0,T;L^2(\om_{\pm},\mathbb{M}_n)),
\end{gather*}
and
\[
\lim_{h\to 0^+}
\|\A_{\lda,h,\pm}-\ol{\A}_{h,\pm}\|_{L^2(\om_{\pm}\times(0,T),\mathbb{M}_n)}=0.
\]
Up to a subsequence,
\begin{gather*}
\A_{\lda,h}\to\A\text{ strongly in }L^2(\om_{\pm}\times(0,T),\mathbb{M}_n),\\
\na\A_{\lda,h}\wc^*\na\A\text{ weakly* in }L^{\ift}(0,T;L^2(\om_{\pm},\mathbb{M}_n)),\\
\pa_t\A_{\lda,h}\wc\pa_t\A\text{ weakly in }L^2(\om_{\pm}\times(0,T),\mathbb{M}_n).
\end{gather*}
In particular, $ \A_{\lda,h}\wc\A $ weakly in
$ H^1(\om_{\pm}\times(0,T),\mathbb{M}_n) $. By the trace theorem,
\[
\A_{\lda,h}\to\A\text{ strongly in }L^2(0,T;L^2(\pa\om_{\pm},\mathbb{M}_n)).
\]
Passing to a further subsequence, we may also assume that
\be
\A_{\lda,h}\to\A,\quad \HH^d\text{-a.e. on }\Ga\times(0,T).
\label{AEC}
\ee

For each $ h $, define
\begin{align*}
M_{\lda,h}(\Ga,T)
&:=
\{(x,t)\in\Ga\times(0,T):(\A_{\lda,h,+},\A_{\lda,h,-})(x,t)
\text{ satisfies the minimal pair condition}\},\\
M_{\A}(\Ga,T)
&:=
\{(x,t)\in\Ga\times(0,T):(\A_+,\A_-)(x,t)
\text{ satisfies the minimal pair condition}\}.
\end{align*}
Since $ (\A_{h,+}^m,\A_{h,-}^m) $ is a minimal pair for every $ m $, we have
\[
\HH^d(M_{\lda,h}(\Ga,T))
=
\lda\,\HH^d(\Ga\times(0,T)).
\]
By the basic measure-theoretic property of the limsup of measurable sets,
\be
\HH^d\(\limsup_{N\to+\ift}M_{\lda,h}(\Ga,T)\)
\geq
\limsup_{N\to+\ift}\HH^d(M_{\lda,h}(\Ga,T))=\lda\,\HH^d(\Ga\times(0,T)),
\label{EstimateHausdorff}
\ee
where $ h=\f{T}{N} $. We now show that, up to a $ \HH^d $-null set,
\[
\limsup_{N\to+\ift}M_{\lda,h}(\Ga,T)\subset M_{\A}(\Ga,T).
\]
Define
\[
C_{\A}(\Ga,T)
:=
\{(x,t)\in\Ga\times(0,T):\A_{\lda,h}(x,t)\to\A(x,t)\text{ as }h\to 0^+\}.
\]
By \eqref{AEC}, the complement of $ C_{\A}(\Ga,T) $ has zero
$ \HH^d $-measure. We claim that
\be
C_{\A}(\Ga,T)\cap\(\limsup_{N\to+\ift}M_{\lda,h}(\Ga,T)\)
\subset M_{\A}(\Ga,T).
\label{inclusion}
\ee
Indeed, let
\[
(x,t)\in C_{\A}(\Ga,T)\cap\(\limsup_{N\to+\ift}M_{\lda,h}(\Ga,T)\).
\]
Then there exists a sequence $ \{h_k\}_{k\in\Z_+} $ with $ h_k\to 0^+ $ such
that $ (\A_{\lda,h_k,+},\A_{\lda,h_k,-})(x,t) $ is a minimal pair for every
$ k $. Hence, there are vectors $ \n^{(k)}\in\R^n $, $ |\n^{(k)}|=1 $ such that
\[
\A_{\lda,h_k,+}^{\T}(x,t)\A_{\lda,h_k,-}(x,t)
=
\I-2\n^{(k)}\otimes\n^{(k)}.
\]
Since $ (x,t)\in C_{\A}(\Ga,T) $, we have
\[
\lim_{k\to+\ift}
\A_{\lda,h_k,+}^{\T}(x,t)\A_{\lda,h_k,-}(x,t)
=
\A_+^{\T}(x,t)\A_-(x,t).
\]
Passing to a subsequence if necessary, we may assume that
$ \n^{(k)}\to\n\in\mathbb{S}^{n-1} $. Therefore,
\[
\A_+^{\T}(x,t)\A_-(x,t)=\I-2\n\otimes\n,
\]
which shows that $ (\A_+,\A_-) $ satisfies the minimal pair condition at
$ (x,t) $. This proves \eqref{inclusion}.

Combining \eqref{EstimateHausdorff} and \eqref{inclusion}, and using that
$ \HH^d((\Ga\times(0,T))\setminus C_{\A}(\Ga,T))=0 $, we obtain
\[
\HH^d(M_{\A}(\Ga,T))
\geq
\HH^d\(C_{\A}(\Ga,T)\cap\(\limsup_{N\to+\ift}M_{\lda,h}(\Ga,T)\)\)
\geq
\lda\,\HH^d(\Ga\times(0,T)).
\]
Letting $ \lda\to 1^- $, we conclude that $ (\A_+,\A_-) $ satisfies the
minimal pair condition for $ \HH^d $-a.e. point on $ \Ga\times(0,T) $.

\subsubsection{Showing that the limit \texorpdfstring{$ \A_{\pm} $}{} is a weak solution}

By Remark \ref{RelationWeak}, it suffices to verify that for every
$ \Psi=\Psi_{\pm}\in C_0^{\ift}((\om_{\pm}\cup\Ga)\times(0,T),\mathbb{A}_n) $
satisfying \eqref{BoundaryTest}, the limit $ \A_{\pm} $ satisfies
\eqref{WeakNeumann}. Applying \eqref{WeakEulerLagrange}, we obtain
\[
\sum_{\zeta=\pm}\int_{\om_{\zeta}}
\(\f{1}{h}\((\A_{h,\zeta}^{m+1})^{\T}(\A_{h,\zeta}^{m+1}-\A_{h,\zeta}^m)\):\Psi_{\zeta}
+\((\A_{h,\zeta}^{m+1})^{\T}\na\A_{h,\zeta}^{m+1}\):\na\Psi_{\zeta}\)\ud x=0.
\]
Integrating over $ t\in(mh,(m+1)h] $, we get
\[
\sum_{\zeta=\pm}\int_{mh}^{(m+1)h}\int_{\om_{\zeta}}
\(\f{1}{h}\((\A_{h,\zeta}^{m+1})^{\T}(\A_{h,\zeta}^{m+1}-\A_{h,\zeta}^m)\):\Psi_{\zeta}
+\((\A_{h,\zeta}^{m+1})^{\T}\na\A_{h,\zeta}^{m+1}\):\na\Psi_{\zeta}\)\ud x\ud t=0.
\]
Summing over $ m=0,\dots,N-1 $ and using the definitions of $ \A_h $ and
$ \ol{\A}_h $, we arrive at
\be
\sum_{\zeta=\pm}\int_0^T\int_{\om_{\zeta}}
\((\ol{\A}_{h,\zeta}^{\T}\pa_t\A_{h,\zeta}):\Psi_{\zeta}
+(\ol{\A}_{h,\zeta}^{\T}\na\ol{\A}_{h,\zeta}):\na\Psi_{\zeta}\)\ud x\ud t=0.
\label{eq2}
\ee

We now pass to the limit as $ h\to 0^+ $. We claim that
\be
\begin{aligned}
&\sum_{\zeta=\pm}\int_0^T\int_{\om_{\zeta}}
\((\ol{\A}_{h,\zeta}^{\T}\pa_t\A_{h,\zeta}):\Psi_{\zeta}
+(\ol{\A}_{h,\zeta}^{\T}\na\ol{\A}_{h,\zeta}):\na\Psi_{\zeta}\)\ud x\ud t\\
&\quad\to
\sum_{\zeta=\pm}\int_0^T\int_{\om_{\zeta}}
\((\A_{\zeta}^{\T}\pa_t\A_{\zeta}):\Psi_{\zeta}
+(\A_{\zeta}^{\T}\na\A_{\zeta}):\na\Psi_{\zeta}\)\ud x\ud t.
\end{aligned}
\label{Convergence2}
\ee
Indeed,
\begin{align*}
&\sum_{\zeta=\pm}\int_0^T\int_{\om_{\zeta}}
(\ol{\A}_{h,\zeta}^{\T}\pa_t\A_{h,\zeta}):\Psi_{\zeta}\ud x\ud t
-
\sum_{\zeta=\pm}\int_0^T\int_{\om_{\zeta}}
(\A_{\zeta}^{\T}\pa_t\A_{\zeta}):\Psi_{\zeta}\ud x\ud t\\
&\quad=
\sum_{\zeta=\pm}\int_0^T\int_{\om_{\zeta}}
\((\ol{\A}_{h,\zeta}-\A_{\zeta})^{\T}\pa_t\A_{h,\zeta}\):\Psi_{\zeta}\ud x\ud t\\
&\quad\quad\quad+
\sum_{\zeta=\pm}\int_0^T\int_{\om_{\zeta}}
\(\A_{\zeta}^{\T}(\pa_t\A_{h,\zeta}-\pa_t\A_{\zeta})\):\Psi_{\zeta}\ud x\ud t.
\end{align*}
Since $ \pa_t\A_{h,\pm} $ is uniformly bounded in
$ L^2(\om_{\pm}\times(0,T),\mathbb{M}_n) $ and
$ \ol{\A}_{h,\pm}\to\A_{\pm} $ strongly in
$ L^2(\om_{\pm}\times(0,T),\mathbb{M}_n) $, the first term tends to $ 0 $.
Moreover, since
$ \pa_t\A_{h,\pm}\wc\pa_t\A_{\pm} $ weakly in
$ L^2(\om_{\pm}\times(0,T),\mathbb{M}_n) $ and
$ \A_{\pm}\Psi_{\pm}\in L^2(\om_{\pm}\times(0,T),\mathbb{M}_n) $, the second
term also tends to $ 0 $.

The second term in \eqref{Convergence2} is handled in the same way. Indeed,
$ \na\ol{\A}_{h,\pm}\wc\na\A_{\pm} $ weakly in
$ L^2(\om_{\pm}\times(0,T),\mathbb{M}_n) $ and
$ \ol{\A}_{h,\pm}\to\A_{\pm} $ strongly in
$ L^2(\om_{\pm}\times(0,T),\mathbb{M}_n) $, so
\[
\lim_{h\to 0^+}\int_0^T\int_{\om_{\pm}}
(\ol{\A}_{h,\pm}^{\T}\na\ol{\A}_{h,\pm}):\na\Psi_{\pm}\ud x\ud t
=
\int_0^T\int_{\om_{\pm}}
(\A_{\pm}^{\T}\na\A_{\pm}):\na\Psi_{\pm}\ud x\ud t.
\]
This proves \eqref{Convergence2}. Combining \eqref{eq2} and
\eqref{Convergence2}, we obtain \eqref{WeakNeumann}.

Next, we apply \eqref{WeakEulerLagrange1} to $ \A_{h,\pm}^{m+1} $. Arguing as above, we obtain
\[
\sum_{\zeta=\pm}\int_0^T\int_{\om_{\zeta}}
\((\pa_t\A_{h,\zeta}\ol{\A}_{h,\zeta}^{\T}):\Psi_{\zeta}
+(\na\ol{\A}_{h,\zeta}\ol{\A}_{h,\zeta}^{\T}):\na\Psi_{\zeta}\)\ud x\ud t=0.
\]
Passing to the limit exactly as before yields \eqref{WeakNeumann1}. Hence
$ \A_{\pm} $ is a weak solution of \eqref{HeatSystems}.

It remains to verify the initial condition. Since
$ \A_h\wc\A $ weakly in $ H^1(\om_{\pm}\times(0,T),\mathbb{M}_n) $, we have $ \A(0)=\A_0 $ in the sense of trace. We also want the stronger convergence stated in \eqref{initial1}. By \eqref{Comparison}, for any
$ m\geq 0 $,
\[
\sum_{\zeta=\pm}\int_{\om_{\zeta}}\|\A_{h,\zeta}^{m+1}-\A_{h,\zeta}^m\|^2\ud x\leq
h\(\sum_{\zeta=\pm}\int_{\om_{\zeta}}\|\na\A_{h,\zeta}^m\|^2\ud x\)\leq
h\(\sum_{\zeta=\pm}\int_{\om_{\zeta}}\|\na\A_{0,\zeta}\|^2\ud x\).
\]
Therefore, if $ t\in(mh,(m+1)h] $, then
\[
\sum_{\zeta=\pm}\int_{\om_{\zeta}}\|\ol{\A}_{h,\zeta}(t)-\A_{0,\zeta}\|^2\ud x
\leq
C(t+h)\(\sum_{\zeta=\pm}\int_{\om_{\zeta}}\|\na\A_{0,\zeta}\|^2\ud x\).
\]
Letting $ h\to 0^+ $, and extracting a subsequence if necessary, we obtain
\be
\sum_{\zeta=\pm}\int_{\om_{\zeta}}\|\A_{\zeta}(t)-\A_{0,\zeta}\|^2\ud x
\leq Ct\(\sum_{\zeta=\pm}\int_{\om_{\zeta}}\|\na\A_{0,\zeta}\|^2\ud x\)
\label{ProofInitial}
\ee
for a.e. $ t\in(0,T) $. Since
$ \pa_t\A\in L^2(0,T;L^2(\om_{\pm},\mathbb{M}_n)) $, the map
$ t\mapsto \A(\cdot,t) $ is absolutely continuous as an
$ L^2(\om_{\pm},\mathbb{M}_n) $-valued function. Hence \eqref{ProofInitial}
holds for every $ t\in(0,T) $, and \eqref{initial1} follows.

\subsection{The non-fixed-interface case: proof of Theorem \ref{ExistenceNonFix}}

We now turn to the case where $ \Ga(t) $ moves with time and prove
Theorem \ref{ExistenceNonFix}. By the local existence theory for mean curvature
flow, there exists $ T_0>0 $ such that, for every $ 0<T<T_0 $, the hypersurface
$ \Ga(t) $ remains smooth for all $ t\in[0,T] $. Fix such a $ T\in(0,T_0) $,
and write $ h=\f{T}{N} $ with $ N\in\Z_+ $. If $ N $ is sufficiently large,
equivalently if $ h\leq h_0 $ for some small constant $ h_0>0 $ depending only
on $ \om $, $ T_0 $, and $ \Ga(0) $, then there exist diffeomorphisms
\[
\Phi_h^m(\cdot,t):=\Phi_{h,\pm}^m(\cdot,t)
=(\Phi_{h,\pm}^{m,\al}(\cdot,t))_{\al=1}^d:\om_{\pm}(t)\to\om_{\pm}(mh),
\quad t\in(mh,(m+1)h].
\]
Here $ \al\in\Z\cap[1,d] $ denotes the $ \al^{\op{th}} $ component of
$ \Phi_h^m $. We assume that the family
$ \{\Phi_h^m\}_{m=0}^{N-1} $ satisfies the uniform estimates
\begin{align}
\sup_{m,h}\sup_{t\in(mh,(m+1)h]}
\|D\Phi_{h,\pm}^m(\cdot,t)-\I_d\|_{L^{\ift}(\om_{\pm}(t),\mathbb{M}_d)}
&\leq C_0h,\label{EstimateDPhi}\\
\sup_{m,h}\sup_{\al,\beta,\ell}\sup_{t\in(mh,(m+1)h]}
\|\pa_t^{\ell}\pa_{\beta}\Phi_{h,\pm}^{m,\al}(\cdot,t)\|_{L^{\ift}(\om_{\pm}(t))}
&\leq C_1,\label{EstimatepatPhi}
\end{align}
where $ C_0,C_1>0 $ depend only on $ \om $, $ T_0 $, and $ \Ga(0) $, while
$ \ell=1,2 $ and $ \al,\beta\in\Z\cap[1,d] $.

Since $ \om_{\pm}(t) $ is smooth for $ t\in[0,T] $, we may also define
piecewise maps $ \Phi_h(\cdot,\cdot):\om_{T,\pm}\to\R^d $ by setting
$ \Phi_h(\cdot,t)=\Phi_{h,\pm}^m(\cdot,t) $ for $ t\in(mh,(m+1)h] $. For every
compact set $ K_{T,\pm}\subset\subset\om_{T,\pm} $, we assume that
\be
\sup_{\al,\beta,\ga\in\Z\cap[1,d]}
\|\pa_{\al}\pa_{\beta}\Phi_{h,\pm}^{\ga}\|_{L^{\ift}(K_{T,\pm})}
\leq C_2,
\label{EstimateD2Phi}
\ee
where $ C_2 $ depends only on $ \om $, $ T_0 $, $ \Ga(0) $, and $ K_{T,\pm} $.
For clarity, note that \eqref{EstimateDPhi} means
\[
\sup_{m,h}\sup_{\al,\beta}\sup_{t\in(mh,(m+1)h]}
\|\pa_{\al}\Phi_{h,\pm}^{m,\beta}(\cdot,t)-\delta_{\al\beta}\|_{L^{\ift}(\om_{\pm}(t))}
\leq C_0h,
\]
where $ \delta_{\al\beta}=1 $ if $ \al=\beta $ and $ 0 $ otherwise. For later
use, we also define
\be
\ol{\Phi}_h^m(\cdot,t):=\ol{\Phi}_{h,\pm}^m(\cdot,t)
:=
(\Phi_{h,\pm}^m(\cdot,(m+1)h))^{-1}(\Phi_{h,\pm}^m(\cdot,t))
:\om_{\pm}(t)\to\om_{\pm}((m+1)h).
\label{PhiOl}
\ee

\subsubsection{Construction of approximating solutions}

We now construct $ \wt{\A}_h $ and $ \ol{\A}_h $, following the scheme used in
the fixed-interface case.

\begin{enumerate}[label=$\text{Step \theenumi}.$]
\item Let $ \A_0=\A_{0,\pm}\in\mathcal{A}_{\Ga(0)} $ and set
$ \A_h^0:=\A_0 $.

\item We define $ \A_h^m:=\A_{h,\pm}^m $ inductively for
$ m\in\Z\cap[0,N] $. Assume that $ \A_h^m $ has already been defined. Set
\[
\wt{\A}_h^m(\cdot):=\wt{\A}_{h,\pm}^m(\cdot)
:=
\A_{h,\pm}^m(\Phi_{h,\pm}^m(\cdot,(m+1)h))
\]
and
\[
\mathbf{V}_h^m(\cdot):=\mathbf{V}_{h,\pm}^m(\cdot)
:=
\pa_t\Phi_{h,\pm}^m(\cdot,(m+1)h^-)
:\om_{\pm}((m+1)h)\to\R^d.
\]
We then define $ \A_h^{m+1} $ by
\[
\A_h^{m+1}
:=
\A(\mathbf{V}_h^m,\wt{\A}_h^m;h)
:=
\mathop{\argmin}_{\A\in\mathcal{A}_{\Ga((m+1)h)}}
E_h(\mathbf{V}_h^m,\wt{\A}_h^m;\A).
\]

\item For $ t\in(mh,(m+1)h] $ and $ x\in\om_{\pm}(mh) $, we define
$ \A_h(x,t) $ by linear interpolation between $ \A_h^m $ and the pullback of
$ \A_h^{m+1} $:
\[
\A_h(x,t)=\(1-\f{t-mh}{h}\)\A_h^m(x)
+\f{t-mh}{h}\A_h^{m+1}((\Phi_h^m)^{-1}(x,(m+1)h)).
\]

\item Finally, for $ t\in(mh,(m+1)h] $ and $ x\in\om_{\pm}(t) $, we set
\begin{align*}
\wt{\A}_h(x,t)&:=\A_h(\Phi_h^m(x,t),t),\\
\ol{\A}_h(x,t)&:=\A_h^{m+1}(\ol{\Phi}_h^m(x,t)).
\end{align*}
\end{enumerate}

\subsubsection{Uniform estimates and convergence of \texorpdfstring{$ \wt{\A}_h $}{} and \texorpdfstring{$ \ol{\A}_h $}{}}

For simplicity, set $ t_k:=kh $ for $ k\in\Z\cap[0,N] $. Since $ \A_h^{m+1} $ is defined by a minimizing problem at time $ t_{m+1} $, we may use the transported field $ \A_h^m(\Phi_h^m(\cdot,t_{m+1})) $ as a competitor and obtain
\begin{align*}
&\sum_{\zeta=\pm}\int_{\om_{\zeta}(t_{m+1})}\left(\|\na\A_{h,\zeta}^{m+1}\|^2+h^{-1}\|\A_{h,\zeta}^{m+1}(\cdot)-\A_{h,\zeta}^m(\Phi_{h,\zeta}^m(\cdot,t_{m+1}))\|^2\right)\ud x\\
&\leq\sum_{\zeta=\pm}\int_{\om_{\zeta}(t_{m+1})}\|\na\{\A_{h,\zeta}^m(\Phi_{h,\zeta}^m(\cdot,t_{m+1}))\}\|^2\ud x\\
&\quad-\sum_{\zeta=\pm}\int_{\om_{\zeta}(t_{m+1})}\left\{2\pa_t\Phi_{h,\zeta}^m(\cdot,t_{m+1}^-)\cdot\na\left[\A_{h,\zeta}^m(\Phi_{h,\zeta}^m(\cdot,t_{m+1}))\right]\right\}:\{\A_{h,\zeta}^{m+1}(\cdot)-\A_{h,\zeta}^m(\Phi_{h,\zeta}^m(\cdot,t_{m+1}))\}\ud x,
\end{align*}
where
\begin{align*}
&\pa_t\Phi_{h,\pm}^m(x,t_{m+1}^-)\cdot\na\{\A_{h,\pm}^m(\Phi_{h,\pm}^m(x,t_{m+1}))\}\\
&\quad=\sum_{\al=1}^d\pa_t\Phi_{h,\pm}^{m,\al}(x,t_{m+1}^-)\pa_{\al}\{\A_{h,\pm}^m(\Phi_{h,\pm}^m(x,t_{m+1}))\}.
\end{align*}
Using \eqref{EstimatepatPhi} and Young's inequality,
\[
ab\leq\frac{\delta}{2}a^2+\frac{1}{2\delta}b^2 \quad \text{for } a,b,\delta>0,
\]
we obtain
\be
\begin{aligned}
&\sum_{\zeta=\pm}\int_{\om_{\zeta}(t_{m+1})}\left(\|\na\A_{h,\zeta}^{m+1}\|^2+h^{-1}\|\A_{h,\zeta}^{m+1}(\cdot)-\A_{h,\zeta}^m(\Phi_{h,\zeta}^m(\cdot,t_{m+1}))\|^2\right)\ud x\\
&\quad\leq\sum_{\zeta=\pm}\int_{\om_{\zeta}(t_{m+1})}(1+Ch)\|\na\{\A_{h,\zeta}^m(\Phi_{h,\zeta}^m(\cdot,t_{m+1}))\}\|^2\ud x\\
&\quad\quad\quad+\sum_{\zeta=\pm}\int_{\om_{\zeta}(t_{m+1})}(2h)^{-1}\|\A_{h,\zeta}^{m+1}(\cdot)-\A_{h,\zeta}^m(\Phi_{h,\zeta}^m(\cdot,t_{m+1}))\|^2\ud x.
\end{aligned}\label{CauchyFirst}
\ee
By the chain rule,
\begin{align*}
\na\{\A_{h,\pm}^m(\Phi_{h,\pm}^m(x,t_{m+1}))\}
=\{(\na\A_{h,\pm}^{m})(\Phi_{h,\pm}^m(x,t_{m+1}))\}D\Phi_{h,\pm}^m(x,t_{m+1}).
\end{align*}
Combining this identity with \eqref{EstimateDPhi} and \eqref{CauchyFirst}, we infer that
\be
\begin{aligned}
&\sum_{\zeta=\pm}\int_{\om_{\zeta}(t_{m+1})}\left(\|\na\A_{h,\zeta}^{m+1}\|^2+(2h)^{-1}\|\A_{h,\zeta}^{m+1}(\cdot)-\A_{h,\zeta}^m(\Phi_{h,\zeta}^m(\cdot,t_{m+1}))\|^2\right)\ud x\\
&\quad\leq\sum_{\zeta=\pm}\int_{\om_{\zeta}(t_m)}(1+Ch)\|\na\A_{h,\zeta}^m(\cdot)D\Phi_{h,\zeta}^m((\Phi_{h,\zeta}^m)^{-1}(\cdot,t_{m+1}),t_{m+1})\|^2\op{J}((\Phi_{h,\zeta}^m)^{-1}(\cdot,t_{m+1}))\ud x\\
&\quad\leq(1+\wt{C}h)\left(\sum_{\zeta=\pm}\int_{\om_{\zeta}(t_m)}\|\na\A_{h,\zeta}^m\|^2\ud x\right),
\end{aligned}\label{Comparison2}
\ee
where $ \op{J}(\cdot) $ denotes the Jacobi determinant. Here we used the change-of-variables formula together with \eqref{EstimateDPhi} and \eqref{EstimatepatPhi}. The constant $ \wt{C} $ depends only on $ \om $, $ T_0 $, and $ \Ga(0) $. From \eqref{Comparison2}, we deduce that
\be
\begin{aligned}
&e^{-\wt{C}t_{m+1}}\left\{\sum_{\zeta=\pm}\int_{\om_{\zeta}(t_{m+1})}\left(\|\na\A_{h,\zeta}^{m+1}\|^2+\frac{h}{2}\left\|\frac{\A_{h,\zeta}^{m+1}(\cdot)-\A_{h,\zeta}^m(\Phi_{h,\zeta}^m(\cdot,t_{m+1}))}{h}\right\|^2\right)\ud x\right\}\\
&\quad\leq e^{-\wt{C}t_{m+1}}(1+\wt{C}h)\left(\sum_{\zeta=\pm}\int_{\om_{\zeta}(t_m)}\|\na\A_{h,\zeta}^m\|^2\ud x\right)\\
&\quad\leq e^{-\wt{C}t_m}\left(\sum_{\zeta=\pm}\int_{\om_{\zeta}(t_m)}\|\na\A_{h,\zeta}^m\|^2\ud x\right),
\end{aligned}\label{Comparison3}
\ee
where the last inequality follows from $ e^{\wt{C}h}\geq 1+\wt{C}h $. Therefore,
\be
\begin{aligned}
&\sum_{\zeta=\pm}\int_{\om_{\zeta}(t_{m+1})}\frac{h}{2}\left\|\frac{\A_{h,\zeta}^{m+1}(\cdot)-\A_{h,\zeta}^m(\Phi_{h,\zeta}^m(\cdot,t_{m+1}))}{h}\right\|^2\ud x\\
&\leq\sum_{\zeta=\pm}e^{\wt{C}T}\left(e^{-\wt{C}t_m}\int_{\om_{\zeta}(t_m)}\|\na\A_{h,\zeta}^m\|^2\ud x-e^{-\wt{C}t_{m+1}}\int_{\om_{\zeta}(t_{m+1})}\|\na\A_{h,\zeta}^{m+1}\|^2\ud x\right).
\end{aligned}\label{Apriori1}
\ee
Iterating \eqref{Comparison3}, we obtain
\be
\begin{aligned}
&e^{-\wt{C}t_{m+1}}\left(\sum_{\zeta=\pm}\int_{\om_{\zeta}(t_{m+1})}\|\na\A_{h,\zeta}^{m+1}\|^2\ud x\right)\leq e^{-\wt{C}t_m}\left(\sum_{\zeta=\pm}\int_{\om_{\zeta}(t_m)}\|\na\A_{h,\zeta}^m\|^2\ud x\right)\\
&\quad\Ra e^{-\wt{C}t_m}\left(\sum_{\zeta=\pm}\int_{\om_{\zeta}(t_m)}\|\na\A_{h,\zeta}^m\|^2\ud x\right)\leq\ldots\leq\sum_{\zeta=\pm}\int_{\om_{\zeta}(0)}\|\na\A_{0,\zeta}\|^2\ud x\\
&\quad\Ra \sum_{\zeta=\pm}\int_{\om_{\zeta}(t_m)}\|\na\A_{h,\zeta}^m\|^2\ud x\leq e^{\wt{C}T}\left(\sum_{\zeta=\pm}\int_{\om_{\zeta}(0)}\|\na\A_{0,\zeta}\|^2\ud x\right).
\end{aligned}\label{Apriori2}
\ee
Combining \eqref{Apriori1} and \eqref{Apriori2}, we arrive at
\be
\sum_{\zeta=\pm}\int_{\om_{\zeta}(t_{m+1})}h\left\|\frac{\A_{h,\zeta}^{m+1}(\cdot)-\A_{h,\zeta}^m(\Phi_{h,\zeta}^m(\cdot,t_{m+1}))}{h}\right\|^2\ud x\leq Ce^{\wt{C}T}\left(\sum_{\zeta=\pm}\int_{\om_{\zeta}(0)}\|\na\A_{0,\zeta}\|^2\ud x\right).\label{Apriori21}
\ee
The constants $ C $ in \eqref{Apriori2} and \eqref{Apriori21} depend only on $ \om $, $ \Ga(0) $, and $ T_0 $. Applying the chain rule once more, we compute for $ t\in(t_m,t_{m+1}] $ that
\be
\begin{aligned}
\pa_t\wt{\A}_h(x,t)&=\frac{1}{h}\left\{\A_h^{m+1}(\ol{\Phi}_h^m(x,t))-\A_h^m(\Phi_h^m(x,t))\right\}\\
&\quad+\left(\frac{t_{m+1}-t}{h}\right)\pa_t\Phi_h^m(x,t)\cdot\{(\na\A_h^m)(\Phi_h^m(x,t))\}\\
&\quad+\left(\frac{t-t_m}{h}\right)\pa_t\Phi_h^m(x,t)\cdot\left\{\left\{\na\left[\A_h^{m+1}((\Phi_h^m)^{-1}(\cdot,t_{m+1}))\right]\right\}(\Phi_h^m(x,t))\right\},
\end{aligned}\label{Formula2}
\ee
where $ \ol{\Phi}_h^m $ is defined in \eqref{PhiOl}. For later use, we write
\[
\pa_t\wt{\A}_h(x,t):=\frac{1}{h}\left\{\A_h^{m+1}((\Phi_h^m)^{-1}(\Phi_h^m(x,t),t_{m+1}))-\A_h^m(\Phi_h^m(x,t))\right\}+\mathbf{C}_h^m(x,t)+\mathbf{D}_h^m(x,t)
\]
for $ t\in(t_m,t_{m+1}] $, where
\be
\begin{aligned}
\mathbf{C}_h^m(x,t)&:=\left(\frac{t_m-t}{h}\right)\pa_t\Phi_h^m(x,t)\cdot\left\{\left\{(\na\A_h^m)-\left\{\na\left[\A_h^{m+1}((\Phi_h^m)^{-1}(\cdot,t_{m+1}))\right]\right\}\right\}|_{(\Phi_h^m(x,t))}\right\},\\
\mathbf{D}_h^m(x,t)&:=\pa_t\Phi_h^m(x,t)\cdot\{(\na\A_h^m)(\Phi_h^m(x,t))\},
\end{aligned}\label{CDFunction}
\ee
and we define
\[
\mathbf{C}_h(x,t):=\mathbf{C}_h^m(x,t),\quad \mathbf{D}_h(x,t):=\mathbf{D}_h^m(x,t)\quad\text{for } t\in(t_m,t_{m+1}].
\]
Using \eqref{Formula2}, the change-of-variable formula, and the bounds in \eqref{EstimateDPhi} and \eqref{EstimatepatPhi}, we obtain for every $ t\in(t_m,t_{m+1}] $ that
\be
\begin{aligned}
&\sum_{\zeta=\pm}\int_{\om_{\zeta}(t)}\|\pa_t\wt{\A}_{h,\zeta}(\cdot,t)\|^2\ud x\\
&\quad\leq\sum_{\zeta=\pm}\int_{\om_{\zeta}(t)}\left\|\frac{\A_{h,\zeta}^{m+1}(\ol{\Phi}_{h,\zeta}^m(\cdot,t))-\A_{h,\zeta}^m(\Phi_{h,\zeta}^m(\cdot,t))}{h}\right\|^2\ud x\\
&\quad\quad\quad+\sum_{\zeta=\pm}\int_{\om_{\zeta}(t)}\|\pa_t\Phi_{h,\zeta}^m(\cdot,t)\cdot\{(\na\A_{h,\zeta}^m)(\Phi_{h,\zeta}^m(\cdot,t))\}\|^2\ud x\\
&\quad\quad\quad+\sum_{\zeta=\pm}\int_{\om_{\zeta}(t)}\left\|\pa_t\Phi_{h,\zeta}^m(\cdot,t)\cdot\left\{\left\{\na\left[\A_{h,\zeta}^{m+1}((\Phi_{h,\zeta}^m)^{-1}(\cdot,t_{m+1}))\right]\right\}(\Phi_{h,\zeta}^m(\cdot,t))\right\}\right\|^2\ud x\\
&\quad\leq C\left(\sum_{\zeta=\pm}\int_{\om_{\zeta}(t_m)}\left\|\frac{\A_{h,\zeta}^{m+1}((\Phi_{h,\zeta}^m)^{-1}(\cdot,t_{m+1}))-\A_{h,\zeta}^m(\cdot)}{h}\right\|^2\ud x\right)\\
&\quad\quad\quad+C\left(\sum_{\zeta=\pm}\int_{\om_{\zeta}(t_m)}\|\na\A_{h,\zeta}^m\|^2\ud x\right)+C\left(\sum_{\zeta=\pm}\int_{\om_{\zeta}(t_{m+1})}\|\na\A_{h,\zeta}^{m+1}\|^2\ud x\right).
\end{aligned}\label{WeakUse}
\ee
Together with \eqref{Apriori2}, this yields
\begin{align*}
\sum_{\zeta=\pm}\int_{\om_{\zeta}(t)}\|\pa_t\wt{\A}_{h,\zeta}(\cdot,t)\|^2\ud x
&\leq C\left(\sum_{\zeta=\pm}\int_{\om_{\zeta}(t_{m+1})}\left\|\frac{\A_{h,\zeta}^{m+1}(\cdot)-\A_{h,\zeta}^m(\Phi_{h,\zeta}^m(\cdot,t_{m+1}))}{h}\right\|^2\ud x\right)\\
&\quad\quad+Ce^{\wt{C}T}\left(\sum_{\zeta=\pm}\int_{\om_{\zeta}(0)}\|\na\A_{0,\zeta}\|^2\ud x\right).
\end{align*}
Integrating over time and using \eqref{Apriori21}, we conclude that
\be
\begin{aligned}
&\sum_{\zeta=\pm}\int_{0}^{T}\int_{\om_{\zeta}(t)}\|\pa_t\wt{\A}_{h,\zeta}\|^2\ud x\ud t
=\sum_{m=0}^{N-1}\sum_{\zeta=\pm}\int_{t_m}^{t_{m+1}}\int_{\om_{\zeta}(t)}\|\pa_t\wt{\A}_{h,\zeta}\|^2\ud x\ud t\\
&\quad\leq\sum_{m=0}^{N-1}\sum_{\zeta=\pm}e^{\wt{C}T}\left(e^{-\wt{C}t_m}\int_{\om_{\zeta}(t_m)}\|\na\A_{h,\zeta}^m\|^2\ud x-e^{-\wt{C}t_{m+1}}\int_{\om_{\zeta}(t_{m+1})}\|\na\A_{h,\zeta}^{m+1}\|^2\ud x\right)\\
&\quad\quad\quad+\sum_{m=0}^{N-1}Che^{\wt{C}T}\left(\sum_{\zeta=\pm}\int_{\om_{\zeta}(0)}\|\na\A_{0,\zeta}\|^2\ud x\right)\\
&\quad\leq C(1+T)e^{\wt{C}T}\left(\sum_{\zeta=\pm}\int_{\om_{\zeta}(0)}\|\na\A_{0,\zeta}\|^2\ud x\right).
\end{aligned}\label{Boundedness1}
\ee

Again, by the chain rule, for any $ t\in(t_m,t_{m+1}] $, we have
\begin{align*}
\na\ol{\A}_h(x,t)
&=(\na\A_h^{m+1})(\ol{\Phi}_h^m(x,t))\left[D(\Phi_h^m)^{-1}\right](\Phi_h^m(x,t),t_{m+1})(D\Phi_h^m)(x,t).
\end{align*}
Using \eqref{EstimateDPhi}, \eqref{Apriori2}, and the change-of-variables formula, we obtain
\begin{align*}
\sum_{\zeta=\pm}\int_{\om_{\zeta}(t)}\|\na\ol{\A}_{h,\zeta}(\cdot,t)\|^2\ud x
&\leq C\left(\sum_{\zeta=\pm}\int_{\om_{\zeta}(t_{m+1})}\|\na\A_{h,\zeta}^{m+1}\|^2\ud x\right)\\
&\leq Ce^{\wt{C}T}\left(\sum_{\zeta=\pm}\int_{\om_{\zeta}(0)}\|\na\A_{0,\zeta}\|^2\ud x\right),
\end{align*}
and
\begin{align*}
\sum_{\zeta=\pm}\int_{\om_{\zeta}(t)}\|\na\wt{\A}_{h,\zeta}(\cdot,t)\|^2\ud x
&\leq C\sum_{\zeta=\pm}\left(\int_{\om_{\zeta}(t_m)}\|\na\A_{h,\zeta}^m\|^2\ud x+\int_{\om_{\zeta}(t_{m+1})}\|\na\A_{h,\zeta}^{m+1}\|^2\ud x\right)\\
&\leq Ce^{\wt{C}T}\left(\sum_{\zeta=\pm}\int_{\om_{\zeta}(0)}\|\na\A_{0,\zeta}\|^2\ud x\right)
\end{align*}
for any $ t\in(t_m,t_{m+1}] $. Combining these estimates with \eqref{Boundedness1}, we find that
\begin{gather*}
\{\wt{\A}_h\}\text{ is bounded in }L^{\ift}(0,T;H^1(\om_{\pm}(\cdot),\mathbb{M}_n)),\\
\{\pa_t\wt{\A}_h\}\text{ is bounded in }L^2(0,T;L^2(\om_{\pm}(\cdot),\mathbb{M}_n)),\\
\{\ol{\A}_h\}\text{ is bounded in }L^{\ift}(0,T;H^1(\om_{\pm}(\cdot),\mathbb{M}_n)).
\end{gather*}
For Banach spaces $ X_1 $ and $ X_2 $, we define
\[
X_1(0,T;X_2(\om_{\pm}(\cdot),\mathbb{M}_n))
:=\left\{\mathbf{Y}=\mathbf{Y}_{\pm}:\om_{T,\pm}\to\mathbb{M}_n:\left\|\|\mathbf{Y}(\cdot,t)\|_{X_2(\om_{\pm}(t))}\right\|_{X_1(0,T)}<\ift\right\}.
\]
Therefore, there exist $ \A=\A_{\pm} $ and $ \mathbf{B}=\mathbf{B}_{\pm} $ on $ \om_{T,\pm} $ such that
\begin{gather*}
\A\in L^2(0,T;L^2(\om_{\pm}(\cdot),\mathbb{M}_n)),\quad
\na\A\in L^{\ift}(0,T;L^2(\om_{\pm}(\cdot),\mathbb{M}_n)),\\
\pa_t\A\in L^2(0,T;L^2(\om_{\pm}(\cdot),\mathbb{M}_n)),\quad
\mathbf{B}\in L^{\ift}(0,T;H^1(\om_{\pm}(\cdot),\mathbb{M}_n)),
\end{gather*}
and up to a subsequence as $ h\to 0^+ $,
\be
\begin{gathered}
\wt{\A}_h\to\A\text{ strongly in }L^2(0,T;L^2(\om_{\pm}(\cdot),\mathbb{M}_n)),\\
\na\wt{\A}_h\wc^*\na\A\text{ weak* in }L^{\ift}(0,T;L^2(\om_{\pm}(\cdot),\mathbb{M}_n)),\\
\pa_t\wt{\A}_h\wc\pa_t\A\text{ weakly in }L^2(0,T;L^2(\om_{\pm}(\cdot),\mathbb{M}_n)),\\
\ol{\A}_h\wc^*\mathbf{B}\text{ weak* in }L^{\ift}(0,T;H^1(\om_{\pm}(\cdot),\mathbb{M}_n)).
\end{gathered}\label{ConvergenceNonFixed}
\ee
We next identify the two limits.

\begin{claim}
We have $ \A=\mathbf{B} $ a.e. in $ \om_{T,\pm} $.
\end{claim}

\begin{proof}
For any $ x\in\om_{\pm}(t) $ and $ t\in(t_m,t_{m+1}] $, direct computation gives
\begin{align*}
&\wt{\A}_h(x,t)-\ol{\A}_h(x,t)\\
&\quad=\left(1-\frac{t-t_m}{h}\right)\A_h^m(\Phi_h^m(x,t))
+\frac{t-t_m}{h}\A_h^{m+1}(\ol{\Phi}_h^m(x,t))
-\A_h^{m+1}(\ol{\Phi}_h^m(x,t))\\
&\quad=\left(\frac{t-t_{m+1}}{h}\right)\left(\A_h^{m+1}(\ol{\Phi}_h^m(x,t))-\A_h^m(\Phi_h^m(x,t))\right).
\end{align*}
Consequently,
\begin{align*}
&\sum_{\zeta=\pm}\int_{0}^{T}\int_{\om_{\zeta}(t)}\|\wt{\A}_{h,\zeta}-\ol{\A}_{h,\zeta}\|^2\ud x\ud t\\
&\quad=\sum_{m=0}^{N-1}\sum_{\zeta=\pm}\int_{t_m}^{t_{m+1}}\int_{\om_{\zeta}(t)}\|\wt{\A}_{h,\zeta}-\ol{\A}_{h,\zeta}\|^2\ud x\ud t\\
&\quad\leq\sum_{m=0}^{N-1}\sum_{\zeta=\pm}\int_{t_m}^{t_{m+1}}\int_{\om_{\zeta}(t)}
\|\A_{h,\zeta}^{m+1}(\ol{\Phi}_{h,\zeta}^m(\cdot,t))-\A_{h,\zeta}^m(\Phi_{h,\zeta}^m(\cdot,t))\|^2\ud x\ud t\\
&\quad\leq\sum_{m=0}^{N-1}\sum_{\zeta=\pm}\int_{t_m}^{t_{m+1}}
\left(Ch^2\int_{\om_{\zeta}(t_{m+1})}\left\|\frac{\A_{h,\zeta}^{m+1}(\cdot)-\A_{h,\zeta}^m(\Phi_{h,\zeta}^m(\cdot,t_{m+1}))}{h}\right\|^2\ud x\right)\ud t\\
&\quad\leq ChTe^{\wt{C}T}\left(\sum_{\zeta=\pm}\int_{\om_{\zeta}(0)}\|\na\A_{0,\zeta}\|^2\ud x\right).
\end{align*}
Here we used the change-of-variables formula in the fourth line and \eqref{Apriori21} in the last line. Letting $ h\to 0^+ $, we obtain
\[
\wt{\A}_h-\ol{\A}_h\to 0
\quad\text{strongly in }L^2(0,T;L^2(\om_{\pm}(\cdot),\mathbb{M}_n)).
\]
Since $ \wt{\A}_h\to\A $ strongly in the same space, it follows that
\be
\ol{\A}_h\to\A\text{ strongly in }L^2(0,T;L^2(\om_{\pm}(\cdot),\mathbb{M}_n)).\label{OLAL2}
\ee
Combining \eqref{OLAL2} with the fourth convergence in \eqref{ConvergenceNonFixed}, we conclude that $ \A=\mathbf{B} $ a.e. in $ \om_{T,\pm} $.
\end{proof}

Moreover, by construction, $ \ol{\A}_{h,\pm}\in\OO_{\pm}(n) $. Passing to a further subsequence if necessary, we may assume that $ \ol{\A}_h\to\A $ a.e. in $ \om_{T,\pm} $. Hence
\[
\A_{\pm}\in\OO_{\pm}(n)\quad\text{a.e. in } \om_{T,\pm}.
\]
In particular,
\[
\A_{\pm}=\mathbf{B}_{\pm}:\om_{T,\pm}\to\OO_{\pm}(n)
\quad\text{a.e.}.
\]

\subsubsection{Proof that \texorpdfstring{$ (\A_+,\A_-) $}{} forms a minimal pair}

We now show that the limit pair inherits the minimal pair property on the moving interface. The argument is close to the fixed-interface case, but we modify the interpolation near each discrete time level.

Fix $ 0<\lda<1 $ and define $ \A_{\lda,h} $ by
\[
\A_{\lda,h}(x,t)=
\left\{
\begin{aligned}
&\frac{t_{m+1}-\lda h-t}{(1-\lda)h}\A_h^m(x)
+\frac{t-t_m}{(1-\lda)h}\A_h^{m+1}((\Phi_h^m)^{-1}(x,t_{m+1}))
&&\text{if } t\in(t_m,t_{m+1}-\lda h],\\
&\A_h^{m+1}((\Phi_h^m)^{-1}(x,t_{m+1}))
&&\text{if } t\in(t_{m+1}-\lda h,t_{m+1}].
\end{aligned}
\right.
\]
Also, let
\[
\wt{\A}_{\lda,h}(x,t):=\A_{\lda,h}(\Phi_h^m(x,t),t).
\]
By the definition of $ \wt{\A}_{\lda,h} $ and the chain rule, we have
\begin{align*}
\pa_t\wt{\A}_{\lda,h}(x,t)
&=\frac{1}{(1-\lda)h}\{\A_h^{m+1}(\ol{\Phi}_h^m(x,t))-\A_h^m(\Phi_h^m(x,t))\}\\
&\quad+\left(\frac{t_{m+1}-\lda h-t}{(1-\lda)h}\right)\pa_t\Phi_h^m(x,t)\cdot\{(\na\A_h^m)(\Phi_h^m(x,t))\}\\
&\quad+\left(\frac{t-t_m}{(1-\lda)h}\right)\pa_t\Phi_h^m(x,t)\cdot
\left\{\left\{\na\left[\A_h^{m+1}((\Phi_h^m)^{-1}(\cdot,t_{m+1}))\right]\right\}(\Phi_h^m(x,t))\right\}
\end{align*}
for $ t\in(t_m,t_{m+1}-\lda h] $ and
\[
\pa_t\wt{\A}_{\lda,h}(x,t)
=\pa_t\Phi_h^m(x,t)\cdot
\left\{\left\{\na\left[\A_h^{m+1}((\Phi_h^m)^{-1}(\cdot,t_{m+1}))\right]\right\}(\Phi_h^m(x,t))\right\}
\]
for $ t\in(t_{m+1}-\lda h,t_{m+1}] $. Applying the change-of-variables formula together with \eqref{EstimateDPhi}, \eqref{EstimatepatPhi}, \eqref{Apriori2}, and \eqref{Apriori21}, we obtain
\begin{align*}
&\sum_{\zeta=\pm}\int_0^T\int_{\om_{\zeta}(t)}\|\pa_t\wt{\A}_{\lda,h,\zeta}(\cdot,t)\|^2\ud x\ud t\\
&\leq\frac{C}{(1-\lda)^2}\sum_{\zeta=\pm}\sum_{m=0}^{N-1}\int_{t_m}^{t_{m+1}}
\int_{\om_{\zeta}(t)}
\left\|\frac{\A_{h,\zeta}^{m+1}(\ol{\Phi}_{h,\zeta}^m(\cdot,t))-\A_{h,\zeta}^m(\Phi_{h,\zeta}^m(\cdot,t))}{h}\right\|^2\ud x\ud t\\
&\quad+C\sum_{m=0}^{N-1}\sum_{\zeta=\pm}\int_{t_m}^{t_{m+1}}
\int_{\om_{\zeta}(t)}
\|\pa_t\Phi_{h,\zeta}^m(\cdot,t)\cdot\{(\na\A_{h,\zeta}^m)(\Phi_{h,\zeta}^m(\cdot,t))\}\|^2\ud x\ud t\\
&\quad+C\sum_{m=0}^{N-1}\sum_{\zeta=\pm}\int_{t_m}^{t_{m+1}}
\int_{\om_{\zeta}(t)}
\left\|\pa_t\Phi_{h,\zeta}^m(\cdot,t)\cdot
\left\{\left\{\na\left[\A_{h,\zeta}^{m+1}((\Phi_{h,\zeta}^m)^{-1}(\cdot,t_{m+1}))\right]\right\}(\Phi_{h,\zeta}^m(\cdot,t))\right\}\right\|^2\ud x\ud t.
\end{align*}
Hence
\be
\sum_{\zeta=\pm}\int_0^T\int_{\om_{\zeta}(t)}\|\pa_t\wt{\A}_{\lda,h,\zeta}(\cdot,t)\|^2\ud x\ud t
\leq C(1+T)e^{\wt{C}T}\left(\sum_{\zeta=\pm}\int_{\om_{\zeta}(0)}\|\na\A_{0,\zeta}\|^2\ud x\right),\label{patA}
\ee
where $ C $ depends only on $ \om $, $ \Ga(0) $, $ T_0 $, and $ \lda $. Similarly,
\be
\sup_{0\leq t\leq T}\sum_{\zeta=\pm}\int_{\om_{\zeta}(t)}\|\na\wt{\A}_{\lda,h,\zeta}(\cdot,t)\|^2\ud x
\leq Ce^{\wt{C}T}\left(\sum_{\zeta=\pm}\int_{\om_{\zeta}(0)}\|\na\A_{0,\zeta}\|^2\ud x\right).\label{nablaA}
\ee
We also compare $ \wt{\A}_{\lda,h} $ with $ \ol{\A}_h $. If $ t\in(t_m,t_{m+1}-\lda h] $, then
\[
\wt{\A}_{\lda,h}(x,t)-\ol{\A}_h(x,t)
=\left(\frac{t-t_{m+1}+\lda h}{(1-\lda)h}\right)
\{\A_h^{m+1}(\ol{\Phi}_h^m(x,t))-\A_h^m(\Phi_h^m(x,t))\},
\]
whereas
\[
\wt{\A}_{\lda,h}(x,t)-\ol{\A}_h(x,t)=0
\quad\text{for } t\in(t_{m+1}-\lda h,t_{m+1}].
\]
Therefore, by \eqref{Apriori21}, as $ h\to 0^+ $,
\be
\sum_{\zeta=\pm}\int_0^T\int_{\om_{\zeta}(t)}\|\wt{\A}_{\lda,h,\zeta}-\ol{\A}_{h,\zeta}\|^2\ud x\ud t
\leq ChTe^{\wt{C}T}\left(\sum_{\zeta=\pm}\int_{\om_{\zeta}(0)}\|\na\A_{0,\zeta}\|^2\ud x\right)\to 0.\label{Aldas}
\ee
Combining \eqref{patA} and \eqref{nablaA}, we deduce that
\be
\begin{gathered}
\{\wt{\A}_{\lda,h}\}\text{ is bounded in }L^{\ift}(0,T;H^1(\om_{\pm}(\cdot),\mathbb{M}_n)),\\
\{\pa_t\wt{\A}_{\lda,h}\}\text{ is bounded in }L^2(0,T;L^2(\om_{\pm}(\cdot),\mathbb{M}_n)).
\end{gathered}\label{ldaBound}
\ee
Using \eqref{ConvergenceNonFixed}, \eqref{Aldas}, and \eqref{ldaBound}, we obtain, up to a subsequence,
\begin{gather*}
\wt{\A}_{\lda,h}\to\A\text{ strongly in }L^2(0,T;L^2(\om_{\pm}(\cdot),\mathbb{M}_n)),\\
\na\wt{\A}_{\lda,h}\wc^*\na\A\text{ weak* in }L^{\ift}(0,T;L^2(\om_{\pm}(\cdot),\mathbb{M}_n)),\\
\pa_t\wt{\A}_{\lda,h}\wc\pa_t\A\text{ weakly in }L^2(0,T;L^2(\om_{\pm}(\cdot),\mathbb{M}_n)).
\end{gather*}
By the trace theorem, it follows, in particular, that
\[
\wt{\A}_{\lda,h}\to\A
\quad\text{strongly in }L^2(0,T;L^2(\Ga(t),\mathbb{M}_n)).
\]

For fixed $ \lda $ and $ h $, define
\begin{align*}
M_{\lda,h}(\Ga_T)&:=\{(x,t)\in\Ga_T:(\wt{\A}_{\lda,h,+},\wt{\A}_{\lda,h,-})\text{ satisfies the minimal pair condition}\},\\
M_{\A}(\Ga_T)&:=\{(x,t)\in\Ga_T:(\A_+,\A_-)\text{ satisfies the minimal pair condition}\}.
\end{align*}
By the construction of $ \wt{\A}_{\lda,h} $ and $ \ol{\A}_h $, and since $ (\ol{\A}_{h,+},\ol{\A}_{h,-}) $ is a minimal pair, we have
\[
\HH^{d}\left(\Ga_T\backslash M_{\lda,h}(\Ga_T)\right)\leq C(1-\lda),
\]
where $ C $ depends only on $ \om $, $ T $, and $ \Ga(0) $. Recalling that $ h=T/N $, we therefore obtain
\[
\HH^{d}\left(\Ga_T\backslash\left(\limsup_{N\to+\ift}M_{\lda,h}(\Ga_T)\right)\right)\leq C(1-\lda).
\]
Moreover, exactly as in the fixed-interface case, strong trace convergence implies that
\[
\limsup_{N\to+\ift}M_{\lda,h}(\Ga_T)\subset M_{\A}(\Ga_T)
\]
up to an $ \HH^d $-null subset of $ \Ga_T $. Consequently, there exists a constant $ c_1>0 $, depending only on $ \om $, $ \Ga(0) $, and $ T_0 $, such that
\[
\HH^{d}\left(\Ga_T\backslash M_{\A}(\Ga_T)\right)\leq c_1(1-\lda).
\]
Letting $ \lda\to 1^- $, we conclude that $ \A $ satisfies the minimal pair condition on $ \Ga_T $.

\subsubsection{Showing that the limit \texorpdfstring{$ \A_{\pm} $}{} is a weak solution}

To prove that $ \A $ is a weak solution of \eqref{HeatSystems}, it is enough to verify \eqref{WeakNeumann} and \eqref{WeakNeumann1} for every
$ \Psi=\Psi_{\pm}\in C_0^{\ift}(\om_{T,\pm}\cup\Ga_T,\mathbb{A}_n) $
satisfying the boundary condition \eqref{BoundaryTest}. As in the proof of Theorem \ref{ExistenceFix}, the argument for \eqref{WeakNeumann1} is the same as that for \eqref{WeakNeumann}, with \eqref{WeakEulerLagrange1} in place of \eqref{WeakEulerLagrange}. We therefore prove only \eqref{WeakNeumann}.

For any $ t\in(t_m,t_{m+1}] $, using identities in \eqref{Frobidentity}, we have
\be
\begin{aligned}
&\sum_{\zeta=\pm}\int_{\om_{\zeta}(t)}
\left((\ol{\A}_{h,\zeta}^{\T}\pa_t\wt{\A}_{h,\zeta}):\Psi_{\zeta}
+(\ol{\A}_{h,\zeta}^{\T}\na\ol{\A}_{h,\zeta}):\na\Psi_{\zeta}\right)\ud x\\
&\quad=\sum_{\zeta=\pm}\int_{\om_{\zeta}(t)}
\frac{\A_{h,\zeta}^{m+1}(\ol{\Phi}_{h,\zeta}^m(\cdot,t))
-\A_{h,\zeta}^m(\Phi_{h,\zeta}^m(\cdot,t))}{h}
:\left\{\A_{h,\zeta}^{m+1}(\ol{\Phi}_{h,\zeta}^m(\cdot,t))\Psi_{\zeta}(\cdot,t)\right\}\ud x\\
&\quad\quad\quad+\sum_{\zeta=\pm}\int_{\om_{\zeta}(t)}
\mathbf{C}_h^m(\cdot,t):
\left\{\A_{h,\zeta}^{m+1}(\ol{\Phi}_{h,\zeta}^m(\cdot,t))\Psi_{\zeta}(\cdot,t)\right\}\ud x\\
&\quad\quad\quad+\sum_{\zeta=\pm}\int_{\om_{\zeta}(t)}
\mathbf{D}_h^m(\cdot,t):
\left\{\A_{h,\zeta}^{m+1}(\ol{\Phi}_{h,\zeta}^m(\cdot,t))\Psi_{\zeta}(\cdot,t)\right\}\ud x\\
&\quad\quad\quad+\sum_{\zeta=\pm}\int_{\om_{\zeta}(t)}
\na\left\{\A_{h,\zeta}^{m+1}(\ol{\Phi}_{h,\zeta}^m(\cdot,t))\right\}
:\left\{\A_{h,\zeta}^{m+1}(\ol{\Phi}_{h,\zeta}^m(\cdot,t))\na\Psi_{\zeta}(\cdot,t)\right\}\ud x\\
&\quad=:I_h^m(1,t)+I_h^m(2,t)+I_h^m(3,t)+I_h^m(4,t),
\end{aligned}\label{Long}
\ee
where $ \mathbf{C}_h^m $ and $ \mathbf{D}_h^m $ are defined in \eqref{CDFunction}. We first reduce the proof to the following statement:
\be
\lim_{h\to 0^+}\sum_{m=0}^{N-1}\int_{t_m}^{t_{m+1}}
\left(I_h^m(1,t)+I_h^m(2,t)+I_h^m(3,t)+I_h^m(4,t)\right)\ud t=0.\label{Claim}
\ee
If \eqref{Claim} holds, then \eqref{Long} gives
\be
\begin{aligned}
&\lim_{h\to 0^+}\sum_{\zeta=\pm}\int_0^T\int_{\om_{\zeta}(t)}
\left((\ol{\A}_{h,\zeta}^{\T}\pa_t\wt{\A}_{h,\zeta}):\Psi_{\zeta}
+(\ol{\A}_{h,\zeta}^{\T}\na\ol{\A}_{h,\zeta}):\na\Psi_{\zeta}\right)\ud x\ud t\\
&\quad=\lim_{h\to 0^+}\sum_{m=0}^{N-1}\int_{t_m}^{t_{m+1}}
\left(I_h^m(1,t)+I_h^m(2,t)+I_h^m(3,t)+I_h^m(4,t)\right)\ud t=0.
\end{aligned}\label{LimitUse}
\ee
Using \eqref{ConvergenceNonFixed} and the same argument as in the proof of \eqref{Convergence2}, we obtain
\[
\begin{aligned}
&\sum_{\zeta=\pm}\int_0^T\int_{\om_{\zeta}(t)}
\left((\ol{\A}_{h,\zeta}^{\T}\pa_t\wt{\A}_{h,\zeta}):\Psi_{\zeta}
+(\ol{\A}_{h,\zeta}^{\T}\na\ol{\A}_{h,\zeta}):\na\Psi_{\zeta}\right)\ud x\ud t\\
&\quad\to
\sum_{\zeta=\pm}\int_0^T\int_{\om_{\zeta}(t)}
\left((\A_{\zeta}^{\T}\pa_t\A_{\zeta}):\Psi_{\zeta}
+(\A_{\zeta}^{\T}\na\A_{\zeta}):\na\Psi_{\zeta}\right)\ud x\ud t
\end{aligned}
\]
as $ h\to 0^+ $. Combining this with \eqref{LimitUse}, we conclude that
\[
\sum_{\zeta=\pm}\int_0^T\int_{\om_{\zeta}(t)}
\left((\A_{\zeta}^{\T}\pa_t\A_{\zeta}):\Psi_{\zeta}
+(\A_{\zeta}^{\T}\na\A_{\zeta}):\na\Psi_{\zeta}\right)\ud x\ud t=0.
\]
This is exactly \eqref{WeakNeumann}. Hence $ \A $ is a weak solution.

It remains to prove \eqref{Claim}. We divide the argument into
\be
\lim_{h\to 0^+}\sum_{m=0}^{N-1}\int_{t_m}^{t_{m+1}}I_h^m(2,t)\ud t=0,\label{Claim2}
\ee
and
\be
\lim_{h\to 0^+}\sum_{m=0}^{N-1}\int_{t_m}^{t_{m+1}}
\left(I_h^m(1,t)+I_h^m(3,t)+I_h^m(4,t)\right)\ud t=0.\label{Claim4}
\ee

\begin{proof}[Proof of \eqref{Claim2}]
We have
\[
\sum_{m=0}^{N-1}\int_{t_m}^{t_{m+1}}I_h^m(2,t)\ud t
=
\sum_{\zeta=\pm}\int_0^T\int_{\om_{\zeta}(t)}
\mathbf{C}_{h,\zeta}(\cdot,t):\left(\ol{\A}_{h,\zeta}(\cdot,t)\Psi_{\zeta}(\cdot,t)\right)\ud x\ud t.
\]
Using \eqref{EstimateDPhi}, \eqref{EstimatepatPhi}, \eqref{Apriori2}, and the change-of-variables formula, we see that $ \mathbf{C}_h $ is bounded in
$ L^2(0,T;L^2(\om_{\pm}(\cdot),\mathbb{M}_n)) $. Indeed,
\begin{align*}
&\sum_{\zeta=\pm}\int_0^T\int_{\om_{\zeta}(t)}\|\mathbf{C}_{h,\zeta}(\cdot,t)\|^2\ud x\ud t\\
&=\sum_{m=0}^{N-1}\sum_{\zeta=\pm}\int_{t_m}^{t_{m+1}}
\int_{\om_{\zeta}(t)}\|\mathbf{C}_{h,\zeta}^m(\cdot,t)\|^2\ud x\ud t\\
&\leq
C\sum_{m=0}^{N-1}\sum_{\zeta=\pm}\int_{t_m}^{t_{m+1}}
\int_{\om_{\zeta}(t)}
\|\pa_t\Phi_{h,\zeta}^m(\cdot,t)\cdot\{(\na\A_{h,\zeta}^m)(\Phi_{h,\zeta}^m(\cdot,t))\}\|^2\ud x\ud t\\
&\quad+
C\sum_{m=0}^{N-1}\sum_{\zeta=\pm}\int_{t_m}^{t_{m+1}}
\int_{\om_{\zeta}(t)}
\left\|\pa_t\Phi_{h,\zeta}^m(\cdot,t)\cdot
\left\{\left\{\na\left[\A_{h,\zeta}^{m+1}\((\Phi_{h,\zeta}^m)^{-1}(\cdot,t_{m+1})\)\right]\right\}
(\Phi_{h,\zeta}^m(\cdot,t))\right\}\right\|^2\ud x\ud t\\
&\leq
C\sum_{m=0}^{N-1}h
\sum_{\zeta=\pm}\left(
\int_{\om_{\zeta}(t_m)}\|\na\A_{h,\zeta}^m\|^2\ud x
+
\int_{\om_{\zeta}(t_{m+1})}\|\na\A_{h,\zeta}^{m+1}\|^2\ud x
\right)\\
&\leq
C\sum_{\zeta=\pm}\int_{\om_{\zeta}(0)}\|\na\A_{0,\zeta}\|^2\ud x,
\end{align*}
where $ C $ depends only on $ \om $, $ \Ga(0) $, and $ T_0 $. Hence, up to a subsequence, there exists
$ \mathbf{C}\in L^2(0,T;L^2(\om_{\pm}(\cdot),\mathbb{M}_n)) $
such that
\[
\mathbf{C}_h\wc \mathbf{C}
\quad\text{weakly in }L^2(0,T;L^2(\om_{\pm}(\cdot),\mathbb{M}_n)).
\]

It is enough to prove that $ \mathbf{C}=\mathbf{O} $. Indeed, since
$ \ol{\A}_h\to\A $ strongly in $ L^2(0,T;L^2(\om_{\pm}(\cdot),\mathbb{M}_n)) $, we have
\begin{align*}
&\sum_{\zeta=\pm}\(
\int_0^T\int_{\om_{\zeta}(t)}
\mathbf{C}_{h,\zeta}:(\ol{\A}_{h,\zeta}\Psi_{\zeta})\ud x\ud t
-
\int_0^T\int_{\om_{\zeta}(t)}
\mathbf{C}_{\zeta}:(\A_{\zeta}\Psi_{\zeta})\ud x\ud t
\)\\
&\quad=
\sum_{\zeta=\pm}\int_0^T\int_{\om_{\zeta}(t)}
\mathbf{C}_{h,\zeta}:(\ol{\A}_{h,\zeta}\Psi_{\zeta}-\A_{\zeta}\Psi_{\zeta})\ud x\ud t\\
&\quad\quad+
\sum_{\zeta=\pm}\int_0^T\int_{\om_{\zeta}(t)}
(\mathbf{C}_{h,\zeta}-\mathbf{C}_{\zeta}):(\A_{\zeta}\Psi_{\zeta})\ud x\ud t\to 0
\end{align*}
as $ h\to 0^+ $. Therefore, if $ \mathbf{C}=\mathbf{O} $, then \eqref{Claim2} follows.

We now prove $ \mathbf{C}=\mathbf{O} $. Let
$ \Xi=\Xi_{\pm}\in C_0^{\ift}(\om_{T,\pm},\mathbb{M}_n) $.
For $ t\in(t_m,t_{m+1}] $, define
\[
\widetilde{\Xi}(x,t)
=
\frac{t_m-t}{h}\,
\Xi((\Phi_h^m)^{-1}(x,t),t)\,
\pa_t\Phi_h^m((\Phi_h^m)^{-1}(x,t),t)\,
\op{J}((\Phi_h^m)^{-1}(x,t)),
\]
that is,
\[
\widetilde{\Xi}(x,t)
=
\left\{
\frac{t_m-t}{h}\,
\Xi((\Phi_h^m)^{-1}(x,t),t)\,
\pa_t\Phi_h^{m,\al}((\Phi_h^m)^{-1}(x,t),t)\,
\op{J}((\Phi_h^m)^{-1}(x,t))
\right\}_{\al=1}^{d}.
\]
Then integration by parts yields
\be
\begin{aligned}
&\left|
\sum_{\zeta=\pm}\int_{\om_{\zeta}(t)}
\mathbf{C}_{h,\zeta}^m(\cdot,t):\Xi_{\zeta}(\cdot,t)\ud x
\right|\\
&\quad\leq\left|\sum_{\zeta=\pm}\int_{\om_{\zeta}(t_m)}
\left\{\A_{h,\zeta}^m(\cdot)-\A_{h,\zeta}^{m+1}\((\Phi_{h,\zeta}^m)^{-1}(\cdot,t_{m+1})\)
\right\}
:\op{div}\{\widetilde{\Xi}\}_{\zeta}(\cdot,t)\ud x
\right|.
\end{aligned}\label{C0C}
\ee
Using \eqref{EstimateDPhi}, \eqref{EstimatepatPhi}, \eqref{EstimateD2Phi}, and the compact support of $ \Xi $, we obtain
\be
\sum_{\zeta=\pm}\int_{\om_{\zeta}(t)}
\|\op{div}\{\widetilde{\Xi}\}_{\zeta}(\cdot,t)\|^2\ud x
\leq \ol{C}\label{EstimatePsi}
\ee
for any $ t\in(t_m,t_{m+1}] $, where $ \ol{C} $ depends only on $ \Xi $, $ \Ga(0) $, $ \om $, and $ T_0 $. Hence, by \eqref{C0C}, \eqref{EstimatePsi}, and Cauchy's inequality,
\begin{align*}
&\left|\sum_{\zeta=\pm}\int_{\om_{\zeta}(t)}
\mathbf{C}_{h,\zeta}^m(\cdot,t):\Xi_{\zeta}(\cdot,t)\ud x
\right|\\
&\quad\leq C\left(\sum_{\zeta=\pm}\int_{\om_{\zeta}(t_m)}
\left\|\A_{h,\zeta}^m(\cdot)-\A_{h,\zeta}^{m+1}\((\Phi_{h,\zeta}^m)^{-1}(\cdot,t_{m+1})\)\right\|^2\ud x\right)^{\f{1}{2}}.
\end{align*}
Therefore,
\begin{align*}
&\left|\sum_{\zeta=\pm}\int_0^T\int_{\om_{\zeta}(t)}
\mathbf{C}_{h,\zeta}(\cdot,t):\Xi_{\zeta}(\cdot,t)\ud x\ud t
\right|\\
&\quad\leq\sum_{m=0}^{N-1}\int_{t_m}^{t_{m+1}}
\left|\sum_{\zeta=\pm}\int_{\om_{\zeta}(t)}
\mathbf{C}_{h,\zeta}^m(\cdot,t):\Xi_{\zeta}(\cdot,t)\ud x
\right|\ud t\\
&\quad\leq C\sum_{m=0}^{N-1}\int_{t_m}^{t_{m+1}}
\left(\sum_{\zeta=\pm}\int_{\om_{\zeta}(t_m)}
\left\|\A_{h,\zeta}^m(\cdot)
-\A_{h,\zeta}^{m+1}\((\Phi_{h,\zeta}^m)^{-1}(\cdot,t_{m+1})\)
\right\|^2\ud x
\right)^{\f{1}{2}}\ud t\\
&\quad\leq CTh^{\f{1}{2}}
\left(\sum_{\zeta=\pm}\int_{\om_{\zeta}(0)}\|\na\A_{0,\zeta}\|^2\ud x
\right)^{\f{1}{2}}\to 0
\end{align*}
as $ h\to 0^+ $, where we used \eqref{Apriori21} in the last step. Since
$ C_0^{\ift}(\om_{T,\pm},\mathbb{M}_n) $ is dense in $ L^2(0,T;L^2(\om_{\pm}(\cdot),\mathbb{M}_n)) $, we conclude that $ \mathbf{C}=\mathbf{O} $. This proves \eqref{Claim2}.
\end{proof}

\begin{proof}[Proof of \eqref{Claim4}]
We now consider the remaining three terms. By the change-of-variables formula,
\be
\begin{aligned}
I_h^m(1,t)
&=
\sum_{\zeta=\pm}\int_{\om_{\zeta}(t_{m+1})}
\frac{\A_{h,\zeta}^{m+1}(\cdot)-\A_{h,\zeta}^m(\Phi_{h,\zeta}^m(\cdot,t_{m+1}))}{h}\\
&\quad\quad:
\left\{
\A_{h,\zeta}^{m+1}(\cdot)
\Psi_{\zeta}\((\ol{\Phi}_{h,\zeta}^m)^{-1}(\cdot,t),t\)
\right\}
\op{J}\((\ol{\Phi}_{h,\zeta}^m)^{-1}(\cdot,t)\)\ud x,
\end{aligned}\label{I1}
\ee
\be
\begin{aligned}
I_h^m(3,t)
&=
\sum_{\zeta=\pm}\int_{\om_{\zeta}(t_{m+1})}
\left\{
(\pa_t\Phi_{h,\zeta}^m)\((\ol{\Phi}_{h,\zeta}^m)^{-1}(\cdot,t),t\)
\cdot
(\na\A_{h,\zeta}^m)\(\Phi_{h,\zeta}^m(\cdot,t_{m+1})\)
\right\}\\
&\quad\quad:
\left\{
\A_{h,\zeta}^{m+1}(\cdot)
\Psi_{\zeta}\((\ol{\Phi}_{h,\zeta}^m)^{-1}(\cdot,t),t\)
\right\}
\op{J}\((\ol{\Phi}_{h,\zeta}^m)^{-1}(\cdot,t)\)\ud x,
\end{aligned}\label{I3}
\ee
and
\be
\begin{aligned}
I_h^m(4,t)
&=
\sum_{\zeta=\pm}\int_{\om_{\zeta}(t_{m+1})}
\na\A_{h,\zeta}^{m+1}(\cdot)
D\ol{\Phi}_{h,\zeta}^m\((\ol{\Phi}_{h,\zeta}^m)^{-1}(\cdot,t),t\)\\
&\quad\quad:
\left\{
\A_{h,\zeta}^{m+1}(\cdot)
\na\Psi_{\zeta}\((\ol{\Phi}_{h,\zeta}^m)^{-1}(\cdot,t),t\)
\right\}
\op{J}\((\ol{\Phi}_{h,\zeta}^m)^{-1}(\cdot,t)\)\ud x.
\end{aligned}\label{I4}
\ee

To connect these terms with the discrete Euler--Lagrange equation, we use the following test function. For each $ \zeta=\pm $, set
\[
\W_{\zeta}(x):=\Psi_{\zeta}\((\ol{\Phi}_{h,\zeta}^m)^{-1}(x,t),t\).
\]
Applying \eqref{WeakEulerLagrange} with
$ \mathbf{V}=\pa_t\Phi_h^m(\cdot,t_{m+1}^-) $ and using identities in \eqref{Frobidentity}, we obtain
\be
\begin{aligned}
0
&=
\sum_{\zeta=\pm}\int_{\om_{\zeta}(t_{m+1})}
\na\A_{h,\zeta}^{m+1}:(\A_{h,\zeta}^{m+1}\na\W_{\zeta})\ud x\\
&\quad+
\sum_{\zeta=\pm}\int_{\om_{\zeta}(t_{m+1})}
\frac{\A_{h,\zeta}^{m+1}
-\A_{h,\zeta}^m(\Phi_{h,\zeta}^m(\cdot,t_{m+1}))}{h}
:(\A_{h,\zeta}^{m+1}\W_{\zeta})\ud x\\
&\quad+
\sum_{\zeta=\pm}\int_{\om_{\zeta}(t_{m+1})}
\left\{
\pa_t\Phi_{h,\zeta}^m(\cdot,t_{m+1}^-)\cdot
\na\left[\A_{h,\zeta}^m(\Phi_{h,\zeta}^m(\cdot,t_{m+1}))\right]
\right\}
:(\A_{h,\zeta}^{m+1}\W_{\zeta})\ud x.
\end{aligned}\label{ELWE}
\ee
Accordingly, define
\begin{align*}
\wt{I}_h^m(1,t)
&:=\sum_{\zeta=\pm}\int_{\om_{\zeta}(t_{m+1})}
\frac{\A_{h,\zeta}^{m+1}(\cdot)-\A_{h,\zeta}^m(\Phi_{h,\zeta}^m(\cdot,t_{m+1}))}{h}:
\left\{\A_{h,\zeta}^{m+1}(\cdot)
\Psi_{\zeta}\((\ol{\Phi}_{h,\zeta}^m)^{-1}(\cdot,t),t\)
\right\}\ud x,
\end{align*}
\begin{align*}
\wt{I}_h^m(3,t)&:=\sum_{\zeta=\pm}\int_{\om_{\zeta}(t_{m+1})}
\left\{(\pa_t\Phi_{h,\zeta}^m)(\cdot,t_{m+1}^-)\cdot
\left\{(\na\A_{h,\zeta}^m)\(\Phi_{h,\zeta}^m(\cdot,t_{m+1})\)
D\Phi_{h,\zeta}^m(\cdot,t_{m+1})
\right\}\right\}\\
&\quad\quad:\left\{\A_{h,\zeta}^{m+1}(\cdot)
\Psi_{\zeta}\((\ol{\Phi}_{h,\zeta}^m)^{-1}(\cdot,t),t\)
\right\}\ud x,
\end{align*}
and
\begin{align*}
\wt{I}_h^m(4,t)
&:=
\sum_{\zeta=\pm}\int_{\om_{\zeta}(t_{m+1})}
\na\A_{h,\zeta}^{m+1}(\cdot):
\left\{
\A_{h,\zeta}^{m+1}(\cdot)
\na\Psi_{\zeta}\((\ol{\Phi}_{h,\zeta}^m)^{-1}(\cdot,t),t\)
D(\ol{\Phi}_{h,\zeta}^m)^{-1}(\cdot,t)
\right\}\ud x.
\end{align*}
Then \eqref{ELWE} yields
\[
\wt{I}_h^m(1,t)+\wt{I}_h^m(3,t)+\wt{I}_h^m(4,t)=0.
\]
Therefore, to prove \eqref{Claim4}, it suffices to show that for each $ j=1,3,4 $,
\be
\lim_{h\to 0^+}\sum_{m=0}^{N-1}\int_{t_m}^{t_{m+1}}
\left(I_h^m(j,t)-\wt{I}_h^m(j,t)\right)\ud t=0.\label{134}
\ee

We begin with the case $ j=1 $. By \eqref{I1},
\begin{align*}
I_h^m(1,t)-\wt{I}_h^m(1,t)
&=
\sum_{\zeta=\pm}\int_{\om_{\zeta}(t_{m+1})}
\frac{\A_{h,\zeta}^{m+1}(\cdot)-\A_{h,\zeta}^m(\Phi_{h,\zeta}^m(\cdot,t_{m+1}))}{h}\\
&\quad\quad:
\left\{
\A_{h,\zeta}^{m+1}(\cdot)
\Psi_{\zeta}\((\ol{\Phi}_{h,\zeta}^m)^{-1}(\cdot,t),t\)
\right\}
\left\{\op{J}\((\ol{\Phi}_{h,\zeta}^m)^{-1}(\cdot,t)\)-1\right\}\ud x.
\end{align*}
Using \eqref{EstimateDPhi}, we have
\be
\left|\op{J}\((\ol{\Phi}_{h,\zeta}^m)^{-1}(\cdot,t)\)-1\right|\leq Ch,\label{Ch1}
\ee
where $ C $ depends only on $ \om $, $ T_0 $, and $ \Ga(0) $. Since
$ \A_{h,\pm}^{m+1}\in\OO_{\pm}(n) $
and
$ \Psi\in C_0^{\ift}(\om_{T,\pm}\cup\Ga_T,\mathbb{M}_n) $,
Cauchy's inequality and \eqref{Apriori21} imply that
\begin{align*}
|I_h^m(1,t)-\wt{I}_h^m(1,t)|
&\leq
C\left(
\sum_{\zeta=\pm}\int_{\om_{\zeta}(t_{m+1})}
\left\|\A_{h,\zeta}^{m+1}(\cdot)-A_{h,\zeta}^m(\Phi_{h,\zeta}^m(\cdot,t_{m+1}))
\right\|^2\ud x
\right)^{\f{1}{2}}\\
&\leq Ch^{\f{1}{2}}
\left(\sum_{\zeta=\pm}\int_{\om_{\zeta}(0)}\|\na\A_{0,\zeta}\|^2\ud x
\right)^{\f{1}{2}}.
\end{align*}
Hence,
\begin{align*}
\sum_{m=0}^{N-1}\int_{t_m}^{t_{m+1}}|I_h^m(1,t)-\wt{I}_h^m(1,t)|\ud t
&\leq
CTh^{\f{1}{2}}
\left(\sum_{\zeta=\pm}\int_{\om_{\zeta}(0)}\|\na\A_{0,\zeta}\|^2\ud x
\right)^{\f{1}{2}}\\
&\leq Ch^{\f{1}{2}}\to 0
\end{align*}
as $ h\to 0^+ $. Thus, \eqref{134} holds for $ j=1 $.

Next, we consider the case $ j=3 $. By \eqref{EstimateDPhi}, \eqref{EstimatepatPhi}, \eqref{EstimateD2Phi}, the mean value theorem, and \eqref{Ch1}, we have
\be
\begin{aligned}
\|(\pa_t\Phi_{h,\pm}^m)((\ol{\Phi}_{h,\pm}^m)^{-1}(\cdot,t),t)-(\pa_t\Phi_{h,\pm}^m)(\cdot,t_{m+1}^-)\|
&\leq Ch,\\
\|\I_d-D\Phi_{h,\pm}^m(\cdot,t_{m+1})\|
&\leq Ch,
\end{aligned}\label{Ch2}
\ee
on $ \om_{\pm}(t_{m+1}) $. Therefore,
\begin{align*}
|I_h^m(3,t)-\wt{I}_h^m(3,t)|
&\leq
C h
\left(
\sum_{\zeta=\pm}\int_{\om_{\zeta}(t_{m+1})}
\|(\na\A_{h,\zeta}^m)(\Phi_{h,\zeta}^m(\cdot,t_{m+1}))\|^2\ud x
\right)^{\f{1}{2}}\\
&\leq
C h
\left(
\sum_{\zeta=\pm}\int_{\om_{\zeta}(t_m)}\|\na\A_{h,\zeta}^m\|^2\ud x
\right)^{\f{1}{2}}\\
&\leq
C h
\left(
\sum_{\zeta=\pm}\int_{\om_{\zeta}(0)}\|\na\A_{0,\zeta}\|^2\ud x
\right)^{\f{1}{2}},
\end{align*}
where we used the change-of-variables formula in the second line and \eqref{Apriori2} in the third line. Consequently,
\begin{align*}
\sum_{m=0}^{N-1}\int_{t_m}^{t_{m+1}}|I_h^m(3,t)-\wt{I}_h^m(3,t)|\ud t
&\leq
CTh
\left(
\sum_{\zeta=\pm}\int_{\om_{\zeta}(0)}\|\na\A_{0,\zeta}\|^2\ud x
\right)^{\f{1}{2}}\\
&\leq Ch\to 0,
\end{align*}
and \eqref{134} holds for $ j=3 $.

Finally, we verify \eqref{134} for $ j=4 $. We rewrite the difference
$ I_h^m(4,t)-\wt{I}_h^m(4,t) $
in a form that isolates the geometric errors coming from the transformation $ \ol{\Phi}_{h,\zeta}^m $. For $ t\in(t_m,t_{m+1}] $, we have
\be
\begin{aligned}
&I_h^m(4,t)-\wt{I}_h^m(4,t)\\
&=\sum_{\zeta=\pm}\int_{\om_{\zeta}(t_{m+1})}
\left\{\na\A_{h,\zeta}^{m+1}(\cdot)\(D\ol{\Phi}_{h,\zeta}^m((\ol{\Phi}_{h,\zeta}^m)^{-1}(\cdot,t),t)-\I_d\)\right\}\\
&\quad\quad:\left\{\A_{h,\zeta}^{m+1}(\cdot)\na\Psi_{\zeta}((\ol{\Phi}_{h,\zeta}^m)^{-1}(\cdot,t),t)\right\}
\op{J}((\ol{\Phi}_{h,\zeta}^m)^{-1}(\cdot,t))\ud x\\
&\quad+\sum_{\zeta=\pm}\int_{\om_{\zeta}(t_{m+1})}
\na\A_{h,\zeta}^{m+1}(\cdot):
\left\{\A_{h,\zeta}^{m+1}(\cdot)\na\Psi_{\zeta}((\ol{\Phi}_{h,\zeta}^m)^{-1}(\cdot,t),t)\right\}
\left\{\op{J}((\ol{\Phi}_{h,\zeta}^m)^{-1}(\cdot,t))-1\right\}\ud x\\
&\quad+\sum_{\zeta=\pm}\int_{\om_{\zeta}(t_{m+1})}
\na\A_{h,\zeta}^{m+1}(\cdot):
\left\{\A_{h,\zeta}^{m+1}(\cdot)\na\Psi_{\zeta}((\ol{\Phi}_{h,\zeta}^m)^{-1}(\cdot,t),t)\(\I_d-D(\ol{\Phi}_{h,\zeta}^m)^{-1}(\cdot,t)\)\right\}\ud x\\
&=:I_h^m(4,1,t)+I_h^m(4,2,t)+I_h^m(4,3,t).
\end{aligned}\label{I4-I4}
\ee
Using \eqref{EstimateDPhi} and \eqref{EstimatepatPhi}, we obtain on $ \om_{\pm}(t_{m+1}) $, for any $ t\in(t_m,t_{m+1}] $,
\be
\begin{aligned}
\|\I_d-D(\ol{\Phi}_{h,\zeta}^m)^{-1}(\cdot,t)\|&\leq Ch,\\
\|D\ol{\Phi}_{h,\zeta}^m((\ol{\Phi}_{h,\zeta}^m)^{-1}(\cdot,t),t)-\I_d\|&\leq Ch.
\end{aligned}\label{Ch3}
\ee
Combining \eqref{Ch1} and \eqref{Ch3}, and using that $ \A_{h,\pm}^{m+1}\in\OO_{\pm}(n) $, we find
\begin{align*}
\sum_{j=1}^3|I_h^m(4,j,t)|
&\leq Ch\sum_{\zeta=\pm}\int_{\om_{\zeta}(t_{m+1})}\|\na\A_{h,\zeta}^{m+1}\|\ud x\\
&\leq Ch\left(\sum_{\zeta=\pm}\int_{\om_{\zeta}(t_{m+1})}\|\na\A_{h,\zeta}^{m+1}\|^2\ud x\right)^{\f12}\\
&\leq Ch\left(\sum_{\zeta=\pm}\int_{\om_{\zeta}(0)}\|\na\A_{0,\zeta}\|^2\ud x\right)^{\f12}.
\end{align*}
Here we used Cauchy's inequality in the second step and \eqref{Apriori2} in the third. Therefore, by \eqref{I4-I4},
\begin{align*}
\sum_{m=0}^{N-1}\int_{t_m}^{t_{m+1}}|I_h^m(4,t)-\wt{I}_h^m(4,t)|\ud t
&\leq \sum_{j=1}^3\sum_{m=0}^{N-1}\int_{t_m}^{t_{m+1}}|I_h^m(4,j,t)|\ud t\\
&\leq CTh\left(\sum_{\zeta=\pm}\int_{\om_{\zeta}(0)}\|\na\A_{0,\zeta}\|^2\ud x\right)^{\f12}\to 0
\end{align*}
as $ h\to 0^+ $. This proves \eqref{134} for $ j=4 $. Hence, the proof of \eqref{134} is complete, and therefore, \eqref{Claim4} follows.
\end{proof}

By combining the proofs of \eqref{Claim2} and \eqref{Claim4}, we obtain \eqref{Claim}. Hence, $ \A_{\pm} $ satisfies \eqref{WeakNeumann} for all admissible test functions. It remains to verify the initial condition \eqref{initial2}. As in the statement of \eqref{initial2}, we work on the fixed reference domains $ \om_{\zeta}(0) $. For $ t\in(0,T) $, let $ \Theta_{h,\zeta}(\cdot,t):\om_{\zeta}(t)\to\om_{\zeta}(0) $ be the discrete backward flow defined by
\[
\Theta_{h,\zeta}(\cdot,t)
:=
\Phi_{h,\zeta}^{0}(\cdot,t_1)\circ
\Phi_{h,\zeta}^{1}(\cdot,t_2)\circ\cdots\circ
\Phi_{h,\zeta}^{m-1}(\cdot,t_m)\circ
\Phi_{h,\zeta}^{m}(\cdot,t),
\qquad t\in(t_m,t_{m+1}],
\]
and set
\[
\wh{\A}_{h,\zeta}(x,t):=\wt{\A}_{h,\zeta}\(\Theta_{h,\zeta}^{-1}(x,t),t\),
\qquad x\in\om_{\zeta}(0).
\]
Then $ \wh{\A}_{h,\zeta}(\cdot,0)=\A_{0,\zeta} $. Fix $ t\in(0,T) $ and $ h>0 $, and let
$ m\in\Z\cap[0,N-1] $ be such that $ t\in(t_m,t_{m+1}] $. By the definition of $ \wt{\A}_h $, the change of variables, and \eqref{Apriori21}, we have
\begin{equation}\label{tm-new}
\begin{aligned}
&\sum_{\zeta=\pm}\int_{\om_{\zeta}(0)}
\|\wh{\A}_{h,\zeta}(\cdot,t)-\wh{\A}_{h,\zeta}(\cdot,t_m)\|^2\ud x\\
&\quad\leq C\(\sum_{\zeta=\pm}\int_{\om_{\zeta}(t_{m+1})}
\|\A_{h,\zeta}^{m+1}(\cdot)-\A_{h,\zeta}^m(\Phi_{h,\zeta}^m(\cdot,t_{m+1}))\|^2\ud x\)\\
&\quad\leq Ch\(\sum_{\zeta=\pm}\int_{\om_{\zeta}(0)}\|\na\A_{0,\zeta}\|^2\ud x\).
\end{aligned}
\end{equation}

Similarly, for each $ j\in\Z\cap[1,m] $,
\begin{equation}\label{jj1-new}
\sum_{\zeta=\pm}\int_{\om_{\zeta}(0)}
\|\wh{\A}_{h,\zeta}(\cdot,t_j)-\wh{\A}_{h,\zeta}(\cdot,t_{j-1})\|^2\ud x
\leq Ch\sum_{\zeta=\pm}\int_{\om_{\zeta}(0)}\|\na\A_{0,\zeta}\|^2\ud x.
\end{equation}

Therefore,
\[
\wh{\A}_{h,\zeta}(\cdot,t)-\A_{0,\zeta}
=
\wh{\A}_{h,\zeta}(\cdot,t)-\wh{\A}_{h,\zeta}(\cdot,t_m)
+\sum_{j=1}^{m}\(\wh{\A}_{h,\zeta}(\cdot,t_j)-\wh{\A}_{h,\zeta}(\cdot,t_{j-1})\).
\]
By Cauchy's inequality, \eqref{tm-new}, and \eqref{jj1-new}, it follows that
\begin{align*}
&\sum_{\zeta=\pm}\int_{\om_{\zeta}(0)}
\|\wh{\A}_{h,\zeta}(\cdot,t)-\A_{0,\zeta}\|^2\ud x\\
&\quad\leq C(m+1)\(\sum_{\zeta=\pm}\int_{\om_{\zeta}(0)}
\|\wh{\A}_{h,\zeta}(\cdot,t)-\wh{\A}_{h,\zeta}(\cdot,t_m)\|^2\ud x\)\\
&\qquad\quad
+C(m+1)\(\sum_{j=1}^{m}\sum_{\zeta=\pm}\int_{\om_{\zeta}(0)}
\|\wh{\A}_{h,\zeta}(\cdot,t_j)-\wh{\A}_{h,\zeta}(\cdot,t_{j-1})\|^2\ud x\)\\
&\quad\leq C(m+1)h\(\sum_{\zeta=\pm}\int_{\om_{\zeta}(0)}\|\na\A_{0,\zeta}\|^2\ud x\)\\
&\quad\leq C(t+h)\(\sum_{\zeta=\pm}\int_{\om_{\zeta}(0)}\|\na\A_{0,\zeta}\|^2\ud x\).
\end{align*}

Now let $ X_{\zeta}(\cdot,t):\om_{\zeta}(0)\to\om_{\zeta}(t) $ be the flow map used in the formulation of \eqref{initial2}, and define
\[
\wh{\A}_{\zeta}(x,t):=\A_{\zeta}(X_{\zeta}(x,t),t),
\qquad x\in\om_{\zeta}(0).
\]
Since $ \wt{\A}_h\to\A $ strongly in
$ L^2(0,T;L^2(\om_{\pm}(\cdot),\mathbb{M}_n)) $,
after pulling back to the fixed domains we also have
\[
\wh{\A}_h\to\wh{\A}
\quad\text{strongly in}\quad
L^2(0,T;L^2(\om_{\pm}(0),\mathbb{M}_n)).
\]
We may therefore let $ h\to0^+ $ and argue exactly as in the fixed-interface case to conclude that
\[
\lim_{t\to0^+}\sum_{\zeta=\pm}\int_{\om_{\zeta}(0)}
\|\wh{\A}_{\zeta}(\cdot,t)-\A_{0,\zeta}\|^2\ud x=0.
\]
This is precisely \eqref{initial2}.
\appendix

\section{Auxiliary geometric facts}

\subsection{Tangent and normal spaces of the manifold of orthogonal matrices}

We first recall the structure of the tangent and normal spaces of $ \OO(n) $. Let $ \A\in\OO(n) $. Then the tangent space of $ \OO(n) $ at $ \A $, denoted by $ T_{\A}\OO(n) $, is given by
\be
T_{\A}\OO(n)=\{\mathbf{X}\in\mathbb{M}_n:\mathbf{X}^{\T}\A+\A^{\T}\mathbf{X}=0\}.\label{TangentSpace}
\ee
Indeed, if $ \mathbf{X}\in T_{\A}\OO(n) $, then there exists a $ C^1 $ curve $ \al:[-\va,\va]\to\OO(n) $ such that $ \al(0)=\A $ and $ \al'(0)=\mathbf{X} $. Since $ \al^{\T}(t)\al(t)=\I $ for all $ t $, differentiation at $ t=0 $ gives
\[
(\al')^{\T}(0)\al(0)+\al^{\T}(0)\al'(0)=0,
\]
that is,
\be
\mathbf{X}^{\T}\A+\A^{\T}\mathbf{X}=0.\label{AXXA}
\ee
Conversely, let $ \mathbf{X}\in\mathbb{M}_n $ satisfy \eqref{AXXA}. Then $ \A^{\T}\mathbf{X}\in\mathbb{A}_n $. Define
\[
\al(t):=\A\exp(t\A^{\T}\mathbf{X}).
\]
Since $ \A^{\T}\mathbf{X} $ is skew-symmetric, $ \exp(t\A^{\T}\mathbf{X})\in\OO(n) $ holds for all $ t $, and hence $ \al(t)\in\OO(n) $. Moreover,
\[
\al'(0)=\A(\A^{\T}\mathbf{X})=\mathbf{X}.
\]
Therefore \eqref{TangentSpace} holds.

We next identify the normal space. Recall that
\[
\mathbb{S}_n:=\{\mathbf{B}\in\mathbb{M}_n:\mathbf{B}^{\T}=\mathbf{B}\}.
\]
Then
\be
N_{\A}\OO(n)=\{\mathbf{X}\in\mathbb{M}_n:\mathbf{X}=\mathbf{B}\A\text{ for some }\mathbf{B}\in\mathbb{S}_n\}.\label{NormalSpace}
\ee
To prove this, we first note that
\[
\dim T_{\A}\OO(n)=\frac{n(n-1)}{2},\quad \dim N_{\A}\OO(n)=\frac{n(n+1)}{2}.
\]
On the other hand,
\[
\dim\{\mathbf{B}\A:\mathbf{B}\in\mathbb{S}_n\}=\dim\mathbb{S}_n=\frac{n(n+1)}{2},
\]
since the map $ \mathbf{B}\mapsto \mathbf{B}\A $ is linear and injective.

It therefore remains to verify orthogonality. Let $ \mathbf{X}\in T_{\A}\OO(n) $ and let $ \mathbf{B}\in\mathbb{S}_n $. By \eqref{TangentSpace}, the matrix $ \A^{\T}\mathbf{X} $ is skew-symmetric. Hence
\[
\mathbf{X}:(\mathbf{B}\A)=\tr(\mathbf{X}^{\T}\mathbf{B}\A)
=\tr\((\A^{\T}\mathbf{X})^{\T}\mathbf{B}\)=0,
\]
because the Frobenius inner product between a skew-symmetric matrix and a symmetric matrix is zero. Since the two spaces have the correct dimensions, \eqref{NormalSpace} follows.

\subsection{Characterization of \eqref{NeumannJump}}

The next lemma shows that the weak transmission conditions imply the pointwise jump relation when the solution is smooth.

\begin{lem}\label{lemNeumann1}
Assume that $ \A $ is smooth and satisfies \eqref{HeatFlow}, \eqref{MinimalPair}, \eqref{WeakNeumann}, and \eqref{WeakNeumann1}. Then \eqref{NeumannJump} holds.
\end{lem}

\begin{proof}
We first derive an identity for $ \A_{\pm}^{\T}\pa_{\nu}\A_{\pm} $. Multiply \eqref{HeatFlow} on the left by $ \A_{\pm}^{\T} $ and test the resulting equation against an arbitrary function $ \Psi\in C_0^{\ift}(\om_{T,\pm}\cup\Ga_T,\mathbb{A}_n) $ satisfying \eqref{BoundaryTest}. Integration by parts yields
\begin{align*}
&\sum_{\zeta=\pm}\int_0^T\int_{\om_{\zeta}(t)}
\((\A_{\zeta}^{\T}\pa_t\A_{\zeta}):\Psi+(\A_{\zeta}^{\T}\na\A_{\zeta}):\na\Psi\)\ud x\ud t\\
&\quad-\int_0^T\int_{\Ga(t)}
(\A_+^{\T}\pa_{\nu}\A_+-\A_-^{\T}\pa_{\nu}\A_-):\Psi\ud\HH^{d-1}\ud t=0.
\end{align*}
By \eqref{WeakNeumann}, the bulk term vanishes. Therefore
\[
\int_0^T\int_{\Ga(t)}
(\A_+^{\T}\pa_{\nu}\A_+-\A_-^{\T}\pa_{\nu}\A_-):\Psi\ud\HH^{d-1}\ud t=0
\]
for all such $ \Psi $. Since differentiating $ \A_{\pm}^{\T}\A_{\pm}=\I $ in the normal direction gives
\[
(\pa_{\nu}\A_{\pm})^{\T}\A_{\pm}+\A_{\pm}^{\T}\pa_{\nu}\A_{\pm}=0,
\]
the matrices $ \A_{\pm}^{\T}\pa_{\nu}\A_{\pm} $ are symmetric skew-symmetric. Hence, the arbitrariness of $ \Psi $ implies
\be
\A_+^{\T}\pa_{\nu}\A_+=\A_-^{\T}\pa_{\nu}\A_-.\label{Neumproperty0}
\ee

Next, multiply \eqref{HeatFlow} on the right by $ \A_{\pm}^{\T} $ and test against $ \Psi\in C_0^{\ift}(\om_{T,\pm}\cup\Ga_T,\mathbb{A}_n) $ satisfying \eqref{BoundaryTest}. Integration by parts gives
\begin{align*}
&\sum_{\zeta=\pm}\int_0^T\int_{\om_{\zeta}(t)}
\((\pa_t\A_{\zeta}\A_{\zeta}^{\T}):\Psi+(\na\A_{\zeta}\A_{\zeta}^{\T}):\na\Psi\)\ud x\ud t\\
&\quad-\int_0^T\int_{\Ga(t)}
(\pa_{\nu}\A_+\A_+^{\T}-\pa_{\nu}\A_-\A_-^{\T}):\Psi\ud\HH^{d-1}\ud t=0.
\end{align*}
Using \eqref{WeakNeumann1}, we obtain
\be
\pa_{\nu}\A_+\A_+^{\T}=\pa_{\nu}\A_-\A_-^{\T}.\label{Neumproperty1}
\ee

We now prove that $ \pa_{\nu}\A_+=\pa_{\nu}\A_- $. Set
\[
\Lda:=\A_+^{\T}\pa_{\nu}\A_+=\A_-^{\T}\pa_{\nu}\A_-.
\]
By \eqref{MinimalPair}, there exists $ \n\in\Ss^{n-1} $ such that
\[
\A_+^{\T}\A_-=\I-2\n\otimes\n=:\mathbf{R}.
\]
Since \eqref{Neumproperty1} can be written as
\[
\A_+\Lda\A_+^{\T}=\A_-\Lda\A_-^{\T},
\]
multiplying by $ \A_+^{\T} $ from the left and by $ \A_+ $ from the right yields $ \Lda=\mathbf{R}\Lda\mathbf{R} $. Because $ \mathbf{R}^2=\I $, this is equivalent to $ \Lda \mathbf{R}=\mathbf{R}\Lda $. Thus, $ \Lda $ commutes with the reflection $ R $. Since the eigenspaces of $ R $ are $ \op{span}\{\n\} $ and $ \n^{\perp} $, the commutation relation implies that both subspaces are invariant under $ \Lda $. In particular, $ \Lda\n\in \op{Span}\{\n\} $. On the other hand, $ \Lda $ is skew-symmetric, so
\[
\n\cdot\Lda\n=0.
\]
Hence $ \Lda\n=0 $.

Finally, using $ \pa_{\nu}\A_-=\A_-\Lda $, we obtain
\[
\A_+^{\T}\pa_{\nu}\A_-=\A_+^{\T}\A_-\Lda=\mathbf{R}\Lda.
\]
Therefore
\[
\A_+^{\T}(\pa_{\nu}\A_+-\pa_{\nu}\A_-)=\Lda-\mathbf{R}\Lda=(\I-\mathbf{R})\Lda=2(\n\otimes\n)\Lda=0,
\]
because $ \Lda\n=0 $. Since $ \A_+ $ is invertible, it follows that $ \pa_{\nu}\A_+=\pa_{\nu}\A_- $. This is exactly \eqref{NeumannJump}.
\end{proof}

\section*{Acknowledgments}
The authors would like to thank Prof. Fanghua Lin for enlightening advice and helpful discussions.
The authors are also grateful to the referees for their valuable comments and suggestions.
This work is partially supported by the National Key R\&D Program of China under Grant 2023YFA1008801.
W. Wang (Zhejiang University) is partially supported by NSF of China under Grant No. 12271476.
Z. Zhang is partially supported NSF of China under Grant No. 12288101.

\bibliographystyle{plain}

\end{document}